\documentclass[11pt]{article}

\usepackage[margin=1in]{geometry}
\usepackage{amsmath}
\usepackage{amssymb}
\usepackage{amsthm}
\usepackage{mathtools}
\usepackage{stmaryrd}
\SetSymbolFont{stmry}{bold}{U}{stmry}{m}{n}
\usepackage{nicefrac}
\usepackage{graphicx}
\usepackage{xcolor}
\usepackage{array}
\usepackage{multirow}
\usepackage{cite}
\usepackage{subcaption}
\usepackage{authblk}
\usepackage{iftex}

\usepackage[normalem]{ulem}

\usepackage{hyperref}
\hypersetup{colorlinks=true,linkcolor=blue,citecolor=blue,urlcolor=blue}

\ifpdf
\DeclareGraphicsExtensions{.eps,.pdf,.png,.jpg}
\else
\DeclareGraphicsExtensions{.eps}
\fi
\graphicspath{{figs/}}

\allowdisplaybreaks
\theoremstyle{plain}
\newtheorem{theorem}{Theorem}
\newtheorem{lemma}[theorem]{Lemma}
\newtheorem{corollary}[theorem]{Corollary}
\theoremstyle{remark}
\newtheorem{remark}[theorem]{Remark}

\newcommand{\jump}[1]{\llbracket #1 \rrbracket}
\newcommand{\av}[1]{\{\!\!\{#1\}\!\!\}}

\newcommand{\Ctime}{C_{\mathrm{time}} }
\newcommand{\tend}{t_{\mathrm{end}}}

\renewcommand{\v}{\mathbf{v}}
\renewcommand{\u}{\mathbf{u}}
\newcommand{\w}{\mathbf{w}}
\newcommand{\z}{\mathbf{z}}
\newcommand{\q}{\mathbf{q}}

\newcommand{\F}{\mathbf{F}}
\newcommand{\G}{\mathbf{G}}

\newcommand{\e}{\mathbf{e}}

\newcommand{\E}{\mathcal{E}}

\newcommand{\K}{\mathcal{K}}

\newcommand{\charfun}{\chi}

\newcommand{\dt}{\,\mathrm{d}t}
\newcommand{\bx}{\mathbf{x}}
\newcommand{\hst}{h_{\bx,t}}
\newcommand{\hs}{h_{\bx}}
\newcommand{\htime}{h_t}

\newcommand{\B}{\mathbf{B}}

\newcommand{\M}{\mathbf{M}}

\renewcommand{\P}{\mathbf{P}}

\newcommand{\Grad}{\nabla}
\newcommand{\Curl}{\nabla \times}
\newcommand{\Div}{\nabla \cdot}

\newcommand{\nvec}{\mathbf{n}}
\newcommand{\dx}{\,\mathrm{d}\mathbf{x}}
\newcommand{\dS}{\, \mathrm{d}S}

\newcommand{\Pdiv}{\mathbf{P}_{k+1}^{\operatorname{div}}(\mathcal{T}_h)}

\title{An arbitrarily high-order accurate and locally energy/entropy-stable space--time method for a class of nonlinear wave systems}

\author[1,2]{Michael Dumbser}
\author[3]{Ilaria Perugia}
\author[3]{Enrico Zampa}
\affil[1]{Laboratory of Applied Mathematics, DICAM, University of Trento, via Mesiano 77, 38123 Trento, Italy (\texttt{michael.dumbser@unitn.it})}
\affil[2]{Department of Mathematics, Southern University of Science and Technology, Xueyuan Avenue 1088, 518055 Shenzhen, Guangdong, China}
\affil[3]{Faculty of Mathematics, University of Vienna, Oskar-Morgenstern-Platz 1, 1090 Vienna, Austria (\texttt{ilaria.perugia@univie.ac.at}, \texttt{enrico.zampa@univie.ac.at})}
\date{}

\begin{document}

\maketitle

\begin{abstract}
We present a new class of structure-preserving, arbitrarily high-order accurate space--time schemes using 
unstructured simplicial meshes in space
for a class of nonlinear acoustic wave propagation systems, in which the nonlinearity is encoded in the constitutive relations between the state and dual variables, while the differential operators retain the standard first-order acoustic wave structure. 
At the fully discrete level, the schemes satisfy the following properties: I) local conservation of the state variables; II) a local energy/entropy balance on each space--time element, with a consistent, single-valued numerical energy/entropy flux and nonnegative local energy/entropy dissipation, and exact energy conservation when the stabilization terms are omitted; III) weak preservation of the curl involution of the velocity field, even in the presence of discontinuities and with the shock-capturing terms activated. The key ingredients of the new schemes are: i) a discontinuous Galerkin (DG) discretization in space, in which the dual variables are computed as the $L^2$ projections onto the DG space of the energy derivatives with respect to the conserved variables; 
ii) numerical fluxes and a bubble-weighted, element-local viscous regularization expressed in terms of the dual variables; iii) a continuous--discontinuous Petrov--Galerkin method in time. 
Properties~I)--III) are first proved mathematically and are then verified by suitable numerical tests.
\end{abstract}



\section{Introduction}
Considerable research has been devoted to the development of numerical methods that preserve local or global invariants of partial differential equations (PDE). Historically, this effort has been pursued from two distinct but complementary perspectives.
\paragraph{Invariant-preserving finite element methods}
This line of research has focused predominantly on constructing finite element methods that exactly conserve one or more nonlinear \emph{global} invariants. Prominent examples include the conservation of kinetic energy and helicity for the incompressible Euler equations \cite{RebholzHelicity}, as well as energy, cross-helicity, and magnetic helicity for incompressible magnetohydrodynamics \cite{HuLeeXu21,GawlikGB22, MaoXi25}. A systematic approach to conserve global invariants has recently been proposed in \cite{AndrewsFarrell25}. While the global conservation of a coercive invariant, such as the energy, ensures long-term numerical stability, it may not be enough to capture localized physical phenomena, such as shocks and discontinuities. Notable exceptions that successfully enforce \emph{local} invariant preservation remain relatively scarce in the literature (see, e.g., \cite{BeiraoLopezVacca_HamiltonianWave2d}).

A closely related line of research focuses on the geometric properties of the continuous system, such as symplecticity, Poisson brackets, and gradient flow structures, and on preserving these properties at the discrete level \cite{SanchezSymplecticWave, SanchezSymplecticElastodynamics, SanchezSymplecticMaxwell, SanchezSymplecticSemilinearwave, EggerDissipative, EggerShashkov21, LombardiPagliantini25, AndrewsFarrell26}. Again, these are typically global structures. In the context of Hamiltonian problems, notable exceptions include recent works on multisymplecticity \cite{McLachlanSternMS, SternZampa26, SternZampa27}, a property that can be regarded as a local version of symplecticity, and on port-Hamiltonian structures \cite{DGPH, DGlinearPH}. 

\paragraph{Entropy-stable methods for hyperbolic systems}
For hyperbolic systems of conservation laws, satisfying a discrete entropy inequality is crucial for selecting physically meaningful weak solutions. For discontinuous Galerkin (DG) methods, the first cell entropy inequality for scalar conservation laws with the square entropy is due to Jiang and Shu \cite{JiangShu1994}.  
For systems admitting a physical entropy, such as the Euler equations of compressible gasdynamics, one can instead evolve the entropy density directly and recover conservation of the energy density through
a compatible discretization; see, e.g., recent work on hyperbolic thermodynamically compatible (HTC) schemes \cite{HTCGPR,BustoDumbserMHD,AbgrallBustoDumbserGeneralFramework}. A vast body of literature exists for finite volume (FV) and DG methods. These approaches generally fall into three broad categories:
\begin{itemize}
    \item \emph{Change-of-variable formulations:} First pioneered by Hughes, Franca, and Mallet \cite{HughesFrancaMallet_86I} in the context of continuous finite element methods, and later extended to DG schemes (e.g., \cite{Barth99, HiltebrandMishra_spacetimeDG, ZakerzadehMay}), this approach replaces all occurrences of the conservative variables by their expressions in terms of the dual (entropy) variables. While mathematically elegant, it lets the time derivative act on a nonlinear function of the unknowns, so that the mass matrix formally depends on the solution, giving rise to a so-called nonlinear mass matrix. Moreover, retaining local conservation requires recovering the conservative variables from the entropy variables by solving local nonlinear problems, even when explicit time integrators are employed.
    \item \emph{SBP and specialized flux frameworks:} The second category relies on summation-by-parts (SBP) operators combined with carefully tailored quadrature rules and two-point entropy-conservative numerical fluxes in the sense of Tadmor \cite{Tadmor1}. While highly successful, these methods require delicate algebraic constructions and restrictive  
    quadrature constraints to strictly enforce stability. A non-exhaustive list of such contributions includes \cite{FisherCarpenter2013,CarpenterFisherNielsenFrankel2014,Gassner2013,GassnerWintersKopriva2016,Ortleb2017,ChenShu17,Ranocha17, RanochaOffnerSonar2017,Chan18}. Recently, an entropy-correction artificial viscosity, with a coefficient computed from the local violation of a cell entropy inequality, has been proposed to obtain entropy-stable standard DG methods without summation-by-parts operators and flux differencing \cite{ChanEntropyStableCorrectionDG}. For a comprehensive overview, see~\cite{ChenShuReview} and the references therein.
    \item \emph{HTC schemes:} The category of recent HTC schemes, see e.g. \cite{HTCGPR,AbgrallBustoDumbserGeneralFramework, BustoDumbserNewClass,BustoDumbserMHD,HTCDG}, is a dual approach to the pioneering work of Tadmor \cite{Tadmor1,Tadmor2003}. While the latter evolves the total energy density and obtains the entropy inequality as a consequence, the recent HTC approach evolves the entropy density directly and obtains local total energy conservation as a consequence. In both cases, the flux is decomposed into an exactly compatible inviscid flux and a provably entropy-stable viscous flux, which adds a nonnegative artificial dissipation term. The inviscid numerical flux is either constructed directly by means of a path integral in the thermodynamic dual variables \cite{Tadmor1,HTCGPR}, or it is obtained from a central flux by means of a correction factor chosen to satisfy the required extra conservation law. The path integral approach goes back to the work of Tadmor~\cite{Tadmor1}, while the latter construction, known as the Abgrall correction, was introduced in~\cite{AbgrallCorrection}. Both approaches give exact compatibility of the inviscid flux. 
\end{itemize}
\smallskip

To address some of the limitations of these existing approaches, 
we propose a novel DG method that employs both conservative and dual variables to preserve the relevant invariants \emph{locally} and globally. The spatial discretization does not require special quadrature rules or specific choices of basis functions. Because the time derivative acts on the conservative variables, the mass matrix is the standard, solution-independent one, and no nonlinear problems need to be solved to recover the conserved variables from the dual variables when explicit time integrators are used. 
We construct a stabilization and viscous regularization that strictly dissipates energy/entropy while simultaneously preserving the involutions of the system. For the time discretization, we use a continuous--discontinuous (CG--DG) Petrov--Galerkin method in time. This allows 
all structural properties to
hold exactly at the fully discrete level; in particular, the energy is exactly conserved when the stabilization and viscous regularization terms are omitted. This time discretization is implicit and requires the solution of a nonlinear system on each space--time slab. With explicit Runge--Kutta schemes, by contrast, the method remains free of nonlinear solves and retains local conservation and the curl involution, but not the exact energy/entropy balance.

As a model problem, we consider a class of nonlinear acoustic wave propagation systems, in which the nonlinearity is confined to the constitutive relations between the state and dual variables, while the differential operators retain the standard first-order acoustic wave structure. 
These are the simplest nontrivial systems of conservation laws exhibiting all three of the following features: nonlinearity, symmetric hyperbolicity for positive definite energy Hessians, and involutions. This class of systems, together with the nonlinear Maxwell and generalized-Lagrangian-multiplier-augmented Maxwell (Maxwell--GLM) systems, has recently been discretized in~\cite{LuccaDumbser2026} by exactly globally energy-conserving and involution-preserving semi-implicit finite volume schemes of second order in space and time on staggered Cartesian meshes. The present work addresses similar objectives through a different approach, allowing for arbitrarily high order of accuracy in space and time, unstructured meshes in space, local energy balances, and dissipative shock capturing.\medskip

The rest of the paper is organized as follows. In Section~\ref{sec:modelproblem}, we introduce the class of nonlinear acoustic wave propagation systems considered in this work, discuss the associated energy structure, and introduce its viscous regularization, which provides the basis for the dissipative stabilization used in the numerical method. In Section~\ref{sec:discrx}, we introduce the spatial DG discretization with upwind-type fluxes and prove local conservation of the state variables, local and global conservation or dissipation of the energy/entropy, and discrete preservation of the curl involution. In Section~\ref{sec:discrxt}, we complete the method with the CG--DG time discretization and establish the corresponding properties for the fully discrete scheme. Section~\ref{sec:numer} presents numerical results, and Section~\ref{sec:conclusions} concludes the paper.

\section{Governing equations}\label{sec:modelproblem}
Let~$\Omega\subset \mathbb{R}^d$, $d=1,2,3$, be a bounded Lipschitz domain, and~$T>0$ a given final time. In~$Q_T:=\Omega\times (0,T)$, we consider the following nonlinear  acoustic wave propagation system:
\begin{subequations}
\label{eq:nonlinear_acoustics}
    \begin{align}
\label{eq:nonlinear_acoustics_a}
        \partial_t \v + \Grad s &= 0, \\
\label{eq:nonlinear_acoustics_b}
        \partial_t \rho + \Div \u &= 0,
    \end{align}
\end{subequations}
completed with initial conditions and with either periodic boundary conditions or the wall boundary condition $\u\cdot\nvec = 0$ on $\partial\Omega$, where $\nvec$ denotes the outward unit normal on $\partial\Omega$.
Here, $\v$ and~$\rho$ denote the state variables, and~$\u$ and~$s$ the corresponding dual variables, defined through the energy density of the system~$\E=\E(\v,\rho)$ by 
\begin{equation}
\u = \frac{ \partial \E}{\partial \v} \quad \text{and}\quad s= \frac{ \partial \E}{\partial \rho}.
\end{equation}
We assume that~$\E$ is \emph{convex} and, for simplicity, \emph{additively separable}, i.e.
\begin{equation}\label{eq:energy}
\E(\v,\rho) = \E_{\v}(\v) + \E_{\rho}(\rho).
\end{equation}
If the Hessian of $\E$ is positive definite, system~\eqref{eq:nonlinear_acoustics} is a \emph{symmetric hyperbolic} system. 
Separability is only used to motivate the choice of the regularization coefficients in Section~\ref{sec:regularization}; the discrete results of Sections~\ref{sec:discrx} and~\ref{sec:discrxt} hold for a general convex energy, with coefficients $H$ and $\mathbf{H}$ in~\eqref{eq:coeffH} below depending on $(\v,\rho)$.

In particular, when $\E(\v, \rho) = \frac{1}{2}\lvert \v\rvert^2 + \frac{1}{2}c^2\rho^2$, with $c$ a positive constant, the linear wave equation is recovered:
    \begin{align*}
        \partial_t \v + c^2 \Grad \rho &= 0, \\
        \partial_t \rho + \Div \v &= 0.
    \end{align*}
In one space dimension and for $\E(v,\rho)=\frac{1}{2}v^2+g(\rho)$ with $g''>0$, system~\eqref{eq:nonlinear_acoustics} is the $p$-system $\partial_t w-\partial_x v=0$, $\partial_t v+\partial_x p(w)=0$, with $w=-\rho$ and $p(w)=g'(-w)$, so that $p'<0$; its Riemann problem can be solved exactly.

For a separable energy as in~\eqref{eq:energy} and a $d$-dimensional unit vector $\nvec$, the eigenvalues of
the Jacobian of the flux
in the direction $\nvec$ are
\begin{equation}
    \lambda_{1,\ldots,d-1}=0, \qquad \lambda_{d,d+1}=\pm c, \qquad c=\sqrt{\frac{\partial^2 \E}{\partial \rho^2}\,\nvec^T \frac{\partial^2 \E}{\partial \v^2}\,\nvec}.
    \label{eq:eigenvalues}
\end{equation}

If $(\v, \rho)$ is a \emph{strong} solution of~\eqref{eq:nonlinear_acoustics}, multiplying the first equation by~$\u$ and the second by~$s$, taking the sum, and using standard vector calculus identities, we obtain the local energy conservation law in differential form:
\begin{equation}
    \partial_t \E + \Div (s \u) = 0.
    \label{eq:local_energy_cons_diff}
\end{equation}
When $(\v,\rho)$ is only a \emph{weak} solution, the conservation law \eqref{eq:local_energy_cons_diff} generally no longer holds as an equality. Instead, one requests that 
\begin{equation}
\partial_t \E + \Div (s\u) \le 0,
\label{eq:local_entropy_inequality}
\end{equation}
in the sense of distributions, thereby ensuring that the solution is physically relevant. In the terminology of hyperbolic conservation laws, $\E$ is called a (mathematical) \emph{entropy}, $s\u$ an \emph{entropy flux}, and $\u$ and $s$ the corresponding \emph{entropy variables}, also called \emph{main field} by Ruggeri and Strumia~\cite{Ruggeri81}. The inequality in~\eqref{eq:local_entropy_inequality} is referred to as the \emph{entropy inequality}. See e.g. \cite[Chapter 80]{EGIII} and references therein for more details on hyperbolic systems. For this reason, throughout the remainder of the paper we will use the term \emph{energy/entropy} when referring to $\E$. 

Integrating \eqref{eq:local_energy_cons_diff} over a subdomain $\K$ of~$\Omega$ and using the divergence theorem, we obtain the local energy conservation law in integral form:
\begin{equation*}
    \int_{\K} \partial_t \E\dx + \int_{\partial \K}s \u \cdot \nvec_{\partial\K} \dS = 0,
        \label{eq:local_energy_cons_int}
\end{equation*}
where~$\nvec_{\partial\K}$ denotes the unit normal vector on~$\partial \K$ pointing outside~$\K$.

\begin{remark}[Abstract system]\label{rem:abstract_system}
System \eqref{eq:nonlinear_acoustics} is a special case of the abstract system 
\begin{subequations}
\label{eq:diff_forms}
\begin{align}
    \partial_t v + \mathrm{d} s &= 0, 
    \label{eq:diff_forms_a}\\
    \partial_t \rho - \delta u &= 0,
    \label{eq:diff_forms_b}
    \end{align}
\end{subequations}
with $\rho$ and $v$ being differential $k$- and $(k+1)$-forms respectively, $\mathrm{d}$ and $\delta$ the exterior derivative and coderivative (see e.g. \cite{AFW06} for a definition), and 
\begin{equation*}
u = \frac{ \partial \E}{\partial v} \quad \text{and}\quad s= \frac{ \partial \E}{\partial \rho}.
\end{equation*}
The construction and the proofs of this work extend to \eqref{eq:diff_forms}. In particular, they extend to the case $k = 1$, which corresponds to a nonlinear Maxwell system:
    \begin{align*}
    \partial_t \v + \Curl \mathbf{s} &= 0, \\
    \partial_t \boldsymbol{\rho} - \Curl \u &= 0,
    \end{align*}
with $\u = \frac{\partial \E}{\partial \v}$ and $\mathbf{s} = \frac{ \partial \E}{ \partial \boldsymbol{\rho}}$.
Energy-conserving and involution-preserving discretizations of this system and of the nonlinear Maxwell--GLM system were constructed in~\cite{DumbserLuccaPeshkovZanotti2025,LuccaDumbser2026}.
However, for clarity, we confine our presentation to \eqref{eq:nonlinear_acoustics}.
\hfill $\blacksquare$
\end{remark}

\subsection{Viscous regularization}\label{sec:regularization}
System~\eqref{eq:nonlinear_acoustics} is a nonlinear hyperbolic system and may therefore develop shocks. Hence, for both the continuous analysis and the numerical stability, we introduce a viscous regularization. At the same time, the regularization should preserve the key structural properties of the system, i.e. the resulting formulation should dissipate energy/entropy and maintain the involution $\partial_t \Curl \v = 0$. 

 We first consider the naive regularization 
\begin{subequations}
\label{eq:nonlinear_acoustics_regnaive}
\begin{align}
\partial_t \v + \Grad s &= \varepsilon \Delta \v, \label{eq:nonlinear_acoustics_regnaive_v}\\
\partial_t \rho + \Div \u &= \varepsilon \Delta \rho. \label{eq:nonlinear_acoustics_regnaive_rho}
\end{align}
\end{subequations}
Using the identity 
\begin{equation}
    \Delta \v = \Grad \Div \v - \Curl \Curl \v,
    \label{eq:vector_laplacian}
\end{equation}
we obtain that, for periodic boundary conditions, if $\Curl \v = 0$ holds at the initial time, it holds at all times. Moreover, multiplying \eqref{eq:nonlinear_acoustics_regnaive_v} by~$\u$, \eqref{eq:nonlinear_acoustics_regnaive_rho} by $s$, adding the two equations, and integrating by parts over a subdomain $\K$ give
\begin{equation}
    \int_{\K} \partial_t \E \dx + \int_{\partial \K}\left( s \u - \varepsilon ( \nabla \v)^T\u - \varepsilon s\Grad \rho\right)\cdot \nvec_{\partial\K} \dS = -\varepsilon\int_{\K}
    \Grad \v\cdot \Grad \u\dx - \varepsilon \int_{\K}\Grad \rho\cdot \Grad s\dx.
\label{eq:local_energy_cons_reg_naive_int}
\end{equation}
We now study the sign of the right-hand side. 
Using the chain rule, we have
\begin{equation}
\Grad \u = \Grad \frac{\partial \E}{\partial \v} = \frac{\partial^2 \E}{\partial \v^2} \Grad \v,
     \qquad
    \Grad s = \Grad \frac{ \partial \E}{\partial \rho} = 
    \frac{\partial^2 \E}{\partial\rho^2}
    \Grad \rho.
    \label{eq:chain_rule}
\end{equation}
It follows that 
\begin{equation*}
\int_{\K}\Grad \v\cdot \Grad \u\dx = \int_{\K}\Grad \v \cdot\frac{\partial^2 \E}{\partial \v^2} \Grad \v\dx,
     \qquad
     \int_{\K}\Grad \rho\cdot \Grad s\dx = \int_{\K} \Grad \rho\cdot\frac{ \partial^2 \E}
    {\partial\rho^2}
    \Grad \rho\dx.
\end{equation*}
Therefore, if the energy~$\E$ is convex, the right-hand side of \eqref{eq:local_energy_cons_reg_naive_int} is nonpositive, and therefore the viscous regularization correctly dissipates energy locally. 
Leaving aside the case of a non-convex energy, the main problem of~\eqref{eq:nonlinear_acoustics_regnaive} is that the chain rule \eqref{eq:chain_rule} is not satisfied by most numerical methods, and therefore the regularization~\eqref{eq:nonlinear_acoustics_regnaive} is not appealing from a numerical point of view.

The above considerations, together with the identity~\eqref{eq:vector_laplacian}, which shows that on curl-free fields the vector Laplacian reduces to its grad--div part, suggest an alternative regularization:
\begin{subequations}
\label{eq:nonlinear_acoustics_reg}
\begin{align}
\partial_t \v + \Grad s &= \varepsilon \nabla \bigl(H(\v)\Div \u\bigr), \label{eq:nonlinear_acoustics_reg_v}\\
\partial_t \rho + \Div \u &= \varepsilon \Div \bigl(\mathbf{H}(\rho)\Grad s\bigr). \label{eq:nonlinear_acoustics_reg_rho}
\end{align}
\end{subequations}
Here, $H(\v) \in \mathbb{R}$ and $\mathbf{H}(\rho)\in \mathbb{R}^{d\times d}$ denote scalar- and matrix-valued positive (semi-)definite coefficient fields, respectively, which will be specified later. For the wall boundary condition, system \eqref{eq:nonlinear_acoustics_reg} is completed by $\u\cdot\nvec=0$ and $(\mathbf{H}(\rho)\Grad s)\cdot\nvec=0$ on~$\partial\Omega$.

Multiplying \eqref{eq:nonlinear_acoustics_reg_v} by $\u$ and \eqref{eq:nonlinear_acoustics_reg_rho} by $s$, and integrating by parts over a subdomain $\K$, we obtain
\begin{equation*}
\partial_t\int_{\K}  \E \dx
+ \int_{\partial \K}(\F + \G)\cdot \nvec_{\partial\K}\dS = - \Pi_{\K},
\end{equation*}
with the terms $\F$, $\G$ and $\Pi_{\K}$ defined as
\begin{align*}
    \F &=  s \u,\\
    \G &= - \varepsilon \bigl( H(\v)(\Div \u )\u + \mathbf{H}(\rho)s\Grad s \bigr), \\
    \Pi_{\K} &= \varepsilon  \int_{\K}(H(\v)(\Div \u)^2+ \mathbf{H}(\rho) \Grad s\cdot \Grad s)\dx.
\end{align*}
By the positivity assumptions on $H(\v)$ and $\mathbf{H}(\rho)$, we have~$\Pi_{\K} \geq 0$, regardless of the convexity of $\E$, yielding a local energy dissipation inequality; accordingly, $\Pi_{\K}$ is called the local energy/entropy dissipation. 
Finally, taking the $\Curl$ of \eqref{eq:nonlinear_acoustics_reg_v} shows that the involution $\partial_t \Curl \v = 0$ is preserved.

\paragraph{Choice of $H(\v)$ and $\mathbf{H}(\rho)$}
The chain rule \eqref{eq:chain_rule} suggests that, if the Hessian of the energy is positive definite, a natural choice is 
\begin{equation}\label{eq:coeffH}
    H(\v) = \left \lVert \left[ \frac{ \partial^2 \E}{\partial \v^2} \right]^{-1} \right \rVert, \qquad \mathbf{H}(\rho) = \left[\frac{ \partial^2 \E}{\partial \rho^2}\right]^{-1} \mathbf{I}.
\end{equation}
with $\lVert \cdot \rVert$ being any induced matrix norm and~$\mathbf{I}$ the identity tensor in~$\mathbb{R}^d$. In fact, if the Hessian of $\E$ is positive definite, so is its inverse, and therefore $H(\v)>0$ and $\mathbf{H}(\rho)$ is positive definite.

\subsection{Global Hamiltonian point of view}
The regularized system \eqref{eq:nonlinear_acoustics_reg} can be written as a Lie--Poisson system with gradient flow dissipation:
\begin{equation*}
    \partial_t \begin{pmatrix} \v \\ \rho \end{pmatrix} = \B \begin{pmatrix} \nicefrac{ \partial \E}{\partial \v} \\ \nicefrac{ \partial \E}{\partial \rho} \end{pmatrix} + \M(\v, \rho) \begin{pmatrix} \nicefrac{ \partial \E}{\partial \v} \\ \nicefrac{ \partial \E}{\partial \rho} \end{pmatrix}
\end{equation*}
Here $\B$ is the Lie--Poisson operator $
    \B \coloneqq \begin{pmatrix} & - \Grad \\ -\Div & \end{pmatrix},
$
which, for periodic or wall boundary conditions, is skew-symmetric and satisfies the Jacobi identity, while the gradient flow operator 
\begin{equation*}
    \M(\v, \rho): \begin{pmatrix} \w \\ r\end{pmatrix} \mapsto \begin{pmatrix} \varepsilon\Grad (H(\v) \Div \w) \\ \varepsilon \Div(\mathbf{H}(\rho) \Grad r )\end{pmatrix}
\end{equation*}
is a negative semi-definite operator. Within this framework, global energy/entropy balance is a straightforward consequence of the skew-symmetry of $\B$ and the negative semi-definiteness of $\M(\v, \rho)$. 

Regarding discretization, substantial literature has focused first on Hamiltonian and Poisson problems \cite{SanchezSymplecticWave, SanchezSymplecticElastodynamics, SanchezSymplecticMaxwell, SanchezSymplecticSemilinearwave}, and more recently on systems that incorporate gradient dissipation \cite{EggerDissipative, EggerShashkov21, LombardiPagliantini25, AndrewsFarrell26}. The primary objective of these works has been the development of structure-preserving discretizations that maintain skew-symmetry and negative semi-definiteness at the discrete level, thereby enforcing a global discrete energy/entropy balance. Nevertheless, these studies focus on global structures and global energy/entropy balance laws and often neglect \emph{local} structural properties.

\begin{remark}[On the connection with port-Hamiltonian problems and their DG discretizations]
System~\eqref{eq:nonlinear_acoustics} can also be regarded as a port-Hamiltonian system with the linear Dirac structure $
    \begin{pmatrix}
        0 & \Grad \\
        \Div & 0
    \end{pmatrix}$, 
see \cite{DGPH}. Following the terminology of port-Hamiltonian systems, $\u$ and $s$ are the \emph{effort} variables, while $\Div\u$ and $\Grad s$ constitute the corresponding \emph{flows}, up to the sign convention adopted for the latter. 
DG spatial semidiscretizations of this structure have been studied, for instance, in \cite{DGlinearPH}. The present work differs from these contributions in several aspects. In particular, the Hamiltonian considered here is not restricted to be quadratic and may be fully nonlinear. Moreover, we investigate the preservation of the involution constraint and, more importantly, develop a compatible shock-capturing mechanism suitable for the treatment of nonsmooth solutions, which may naturally arise in nonlinear hyperbolic systems. 
\hfill $\blacksquare$
\end{remark}

\section{Spatial semidiscretization}\label{sec:discrx}

In this section, we discretize  the regularized system~\eqref{eq:nonlinear_acoustics_reg} in space only and study the properties of the associated semidiscrete formulation. In the following, we assume that all spatial and temporal integrals are evaluated exactly.

\subsection{Notation, mesh, forms and polynomial spaces}\label{sec:meshes}
We denote by $(\cdot,\cdot)_{\K}$ the $L^2$ inner product over a subdomain $\K \subseteq \Omega$ and by $\langle\cdot,\cdot\rangle_{\partial\K}$ the $L^2$ inner product over its boundary $\partial\K$.
Let $\mathcal{T}_h$ be a conforming, simplicial mesh covering $\Omega$. The skeleton, i.e. the set of all facets $e$ of $\mathcal{T}_h$, is denoted by $\mathcal{S}_h$. 
We define
\begin{align*}
(\cdot, \cdot)_{\mathcal{T}_h} &\coloneqq \sum_{K \in \mathcal{T}_h}(\cdot, \cdot)_{K},\\
\langle \cdot, \cdot\rangle_{\partial \mathcal{T}_h} &\coloneqq \sum_{K \in \mathcal{T}_h}\langle \cdot, \cdot\rangle_{\partial K}.
\end{align*}

Let $K$ be an element of $\mathcal{T}_h$, and let $e$ be a facet of $K$. We denote by $K_e$ the element of $\mathcal{T}_h$ that shares $e$ with $K$.
On the boundary $ \partial K$ of an element~$K$, we define the \emph{average} and the \emph{jump} of a scalar-valued function $\xi$ in the broken space $H^1(\mathcal{T}_h):=\bigoplus_{K\in\mathcal{T}_h}H^1(K)$ as
\begin{align*}
    \av{\xi}_{|e} & = \begin{cases}\frac{1}{2}(\xi_{|K} + \xi_{|K_e}) \qquad &\text{ if $K_e$ exists,}\\
    \xi_{|K}\qquad &\text{ otherwise;}
    \end{cases} \\
    \jump{\xi}_{|e} &= \begin{cases}(\xi_{|K} - \xi_{|K_e}) \qquad &\text{ if $K_e$ exists,} \\
    0 \qquad & \text{ otherwise.}
    \end{cases}
\end{align*}
We define the \emph{average} and the \emph{jump} of a vector-valued function~$\boldsymbol{\xi}$ in the broken space $H^1(\mathcal{T}_h)^d$ in the same way.
We also define the \emph{cross average} $\av{a,b}_{\times} = \av{a}\av{b}-\frac{1}{4}\jump{a}\jump{b}$. Finally, we set
\begin{equation*}
    \overline{\xi}_{|e} = \begin{cases} \max\{\lVert \xi_{|K} \rVert_{L^\infty(e)}, \lVert \xi_{|K_e} \rVert_{L^{\infty}(e)}\} \qquad & \text{ if $K_e$ exists,}\\
    0 \qquad & \text{ otherwise.}
    \end{cases}
\end{equation*}

Next, we introduce forms associated with the regularization terms. Let~$\varepsilon$ be a nonnegative function, possibly discontinuous across element interfaces. Ideally, $\varepsilon$ should be chosen such that $\varepsilon \approx 0$ in smooth regions and $\varepsilon = O(h)$ in the presence of shocks/discontinuities. We discuss some possible strategies to choose $\varepsilon$ in Section~\ref{sec:numer} below. Furthermore, let~$\alpha_h$ be the unique element in $H^1_0(\Omega)$ such that, for each element $K\in\mathcal{T}_h$, ${\alpha_h}_{|K}\in P_{d+1}(K)$, ${\alpha_h}_{|\partial K} = 0$, and $\alpha_h(\mathbf{x}_K) = 1$, with $\mathbf{x}_K$ the barycenter of $K$ (elementwise bubble function). Then, we define the div--div form
\begin{equation}\label{eq:bh}
        b_h(\v_h; \u_h, \w_h)  \coloneqq (\alpha_h\varepsilon H(\v_h) \Div \u_h, \Div \w_h)_{ \mathcal{T}_h} \qquad \forall \v_h, \u_h, \w_h \in H^1(\mathcal{T}_h)^d,
\end{equation}
and the grad--grad form
\begin{equation}\label{eq:ch}
    c_h(\rho_h; s_h, \eta_h) \coloneqq (\alpha_h\varepsilon\mathbf{H}(\rho_h)\nabla s_h, \nabla \eta_h)_{\mathcal{T}_h} \qquad\forall \rho_h, s_h, \eta_h \in H^1(\mathcal{T}_h). 
\end{equation}
Finally, denoting by~$P_k(K)$ the space of polynomials of degree at most~$k$ on~$K$, we set
\begin{equation*}
P_k(\mathcal{T}_h):=\{\eta_h\in H^1(\mathcal{T}_h): {\eta_h}_{|K}\in P_k(K)\quad \forall K\in \mathcal{T}_h\},\qquad \mathbf{P}_k(\mathcal{T}_h):=P_k(\mathcal{T}_h)^d.
\end{equation*}

\subsection{The semidiscrete method}\label{sec:semidiscrete}
Here and in the following, we denote by~$\nvec$ the unit normal vector on the skeleton of $\mathcal{T}_h$ pointing outward from the element under consideration. The semidiscretization in space of problem~\eqref{eq:nonlinear_acoustics_reg} reads as follows: Find $\v_h, \u_h \in \mathbf{P}_k(\mathcal{T}_h)$, $\rho_h, s_h \in P_k(\mathcal{T}_h)$ satisfying
\begin{subequations}
\label{eq:DG_scheme}
    \begin{alignat}{2}
        (\partial_t \v_h, \w_h)_{\mathcal{T}_h} - (s_h, \Div \w_h)_{\mathcal{T}_h} + \langle \widehat{s}_h, \w_h \cdot \nvec\rangle_{\partial \mathcal{T}_h} & = - b_h(\v_h;\u_h, \w_h) \quad &&\forall \w_h \in \mathbf{P}_k(\mathcal{T}_h), \label{eq:DG_scheme_w} \\
        (\partial_t \rho_h, \eta_h)_{\mathcal{T}_h} - (\u_h, \Grad \eta_h)_{\mathcal{T}_h} + \langle \widehat{ \u_h\cdot \nvec}, \eta_h\rangle_{\partial \mathcal{T}_h} &= - c_h(\rho_h; s_h, \eta_h) \quad &&\forall \eta_h\in P_k(\mathcal{T}_h), \label{eq:DG_scheme_eta}\\
        (s_h, r_h)_{\mathcal{T}_h} - \left( \frac{ \partial \E}{\partial \rho_h}, r_h\right)_{\mathcal{T}_h} &= 0 \quad&&\forall r_h \in {P}_k(\mathcal{T}_h), \label{eq:DG_scheme_r}\\ 
        (\u_h, \z_h)_{\mathcal{T}_h} - \left(\frac{ \partial \E}{\partial \v_h}, \z_h \right)_{\mathcal{T}_h} &= 0 \quad &&\forall \z_h \in \mathbf{P}_k(\mathcal{T}_h). \label{eq:DG_scheme_z}
    \end{alignat}
\end{subequations} 
The numerical fluxes are chosen on the boundary of each element $K$ as follows:
\begin{equation}\label{eq:fluxes}
    \widehat{\u_h \cdot \nvec}\coloneqq \av{\u_h }\cdot \nvec + \beta_h \widetilde{\mathbf{H}(\rho_h)}\jump{s_h}, \qquad \widehat{s}_h \coloneqq \av{s_h} + \beta_h\widetilde{H(\v_h)}\jump{\u_h}\cdot{\nvec},
\end{equation}
with $\beta_h$ being a nonnegative function, single-valued on each facet, possibly depending on $\rho_h$ and $\v_h$, with the harmonic averages
\begin{align*}
    \widetilde{\mathbf{H}(\rho_h)}_{|e}& \coloneqq 2\,\frac{ \big(\nvec_{\partial K}^T \mathbf{H}({\rho_h}_{|\partial K}) \nvec_{\partial K}\big)\big( \nvec_{\partial K_e}^T \mathbf{H}({\rho_h}_{|\partial K_e} )\nvec_{\partial K_e}\big)
    }
    {\nvec_{\partial K}^T \mathbf{H}({\rho_h}_{|\partial K}) \nvec_{\partial K} + \nvec_{\partial K_e}^T \mathbf{H}({\rho_h}_{|\partial K_e}) \nvec_{\partial K_e}},\\
    \widetilde{H(\v_h)}_{|e} &\coloneqq 2\,\frac{ H({\v_h}_{|\partial K})H({\v_h}_{|\partial K_e})}{  H({\v_h}_{|\partial K}) + H({\v_h}_{|\partial K_e})},
\end{align*}
which 
are set to zero if both of their arguments vanish.
The flux $\widehat{s}_h$ is single-valued on each facet, while $\widehat{\u_h \cdot \nvec}$ changes sign when the facet orientation is reversed.
For periodic boundary conditions every facet is an interior facet, whereas for the wall boundary condition we set $\widehat{\u_h\cdot\nvec}\coloneqq 0$ on the facets $e\subset\partial\Omega$, while $\widehat{s}_h = s_h$ holds there by the definition of average and jump on boundary facets. In this way, the wall boundary condition is imposed weakly through the numerical flux.
The average terms ensure consistency while the jump terms provide stabilization, with multiplicative coefficients related to the viscous regularizations.
\begin{remark}[On the choice of the viscous regularization]
Since $\alpha_h$ vanishes on element boundaries, the form $c_h(\rho_h; s_h,\eta_h)$ provides a consistent discretization of
$-\Div\bigl(\alpha_h \varepsilon \mathbf{H}(\rho_h)\Grad s_h\bigr)$.
An alternative way to retain consistency with
$-\Div\bigl(\varepsilon \mathbf{H}(\rho_h)\Grad s_h\bigr)
$ would be to omit $\alpha_h$ and introduce suitable interface contributions, leading for instance to an interior-penalty discretization \cite{ArnoldIP}. Such an approach would still preserve energy stability and involutions. However, it presents several disadvantages:
\begin{itemize}
\item Since $\varepsilon$ and $\mathbf{H}(\rho_h)$ may be discontinuous and exhibit large jumps across element interfaces, appropriate interface averages must be introduced, making the formulation more intricate.
\item The method requires an additional stabilization parameter $\omega$, associated with the penalty term
$\frac{\omega}{h}\langle \jump{s_h}, \jump{\eta_h}\rangle_{\partial\mathcal T_h}$.
\item From an implementation viewpoint, hybridization and static condensation are natural tools for improving efficiency. While the scheme in~\eqref{eq:DG_scheme} requires only a facet unknown representing $\av{s_h}$, an interior-penalty formulation would additionally require facet unknowns associated with $\av{\Grad s_h}$.
\end{itemize}
Analogous considerations apply to the bilinear form $b_h$.
\hfill $\blacksquare$
\end{remark}
\subsection{Properties}
\label{sec:properties}
We start by showing that our discretization is conservative. 
\begin{lemma}[Local conservation]
\label{lem:conservativity}
Let $(\v_h,\u_h, \rho_h,s_h)$ be a solution of \eqref{eq:DG_scheme}. Then on each element $K \in \mathcal{T}_h$, the quantities $\v_h\cdot\e_i$ and $\rho_h$ satisfy local conservation laws. More precisely,
        \begin{align*}
            \int_K \partial_t (\v_h\cdot \e_i)\dx + \int_{\partial K}(\widehat{s}_h \e_i)\cdot \nvec\dS &= 0 \qquad \forall i = 1, \dots, d\\
            \int_K\partial_t \rho_h\dx + \int_{\partial K}\widehat{\u_h\cdot \nvec}\dS &= 0.
        \end{align*}
Here $\{\e_1,\ldots,\e_d\}$ denotes the canonical orthonormal basis of $\mathbb{R}^d$.\end{lemma}
\begin{proof}
The result follows by choosing the test functions $\w_h=\charfun_K\e_i$ and $\eta_h=\charfun_K$ in \eqref{eq:DG_scheme_w} and \eqref{eq:DG_scheme_eta}, respectively, where $\charfun_K$ denotes the characteristic function of the element $K$. Since these functions are constant on $K$ and vanish elsewhere, all volume terms involving spatial derivatives disappear, leaving precisely the stated conservation laws.
\end{proof}
\begin{theorem}[Local energy conservation/dissipation]\label{thm:energy_loc} 
Let $(\v_h,\u_h, \rho_h,s_h)$ be a solution of \eqref{eq:DG_scheme}. Then on each element $K \in \mathcal{T}_h$, the following energy balance holds:
    \begin{equation}
        \partial_t \int_K\E\dx + \int_{\partial K}  (\widehat{\F}_h + \widehat{\G}_h )\cdot \nvec \dS = -\Pi_{K, h} \leq 0,
        \label{eq:energy_balance_loc}
    \end{equation}
    where 
    \begin{align}
        \widehat{\F}_h &= \av{\u_h ,s_h}_{\times}, \label{eq:Fhat} \\
        \widehat{\G}_h &= \beta_h\widetilde{\mathbf{H}(\rho_h)}\jump{s_h} \av{s_h} \nvec + \beta_h\widetilde{H(\v_h)}\left( \jump{\u_h }\cdot \nvec\right) \av{\u_h}, \label{eq:Ghat} \\
        \Pi_{K,h} &= (\alpha_h \varepsilon H(\v_h)\Div \u_h, \Div \u_h)_K + (\alpha_h \varepsilon \mathbf{H}(\rho_h)\Grad s_h, \Grad s_h)_K \nonumber \\
        &\quad + \frac{1}{2}\int_{\partial K}\beta_h\Big(\widetilde{\mathbf{H}(\rho_h)}\jump{s_h}^2 + \widetilde{H(\v_h)}\big(\jump{\u_h}\cdot \nvec\big)^2\Big)\dS. \label{eq:Pi_loc}
    \end{align}
Here $\widehat{\F}_h+\widehat{\G}_h$ is the consistent and single-valued numerical energy/entropy flux, composed of the energy-conservative part $\widehat{\F}_h$ and the dissipative part $\widehat{\G}_h$, and $\Pi_{K,h}\geq 0$ is the local energy/entropy dissipation on $K$, consisting of the viscous contribution in the interior of $K$ and of the contribution of the jump terms on $\partial K$. On wall facets, we set $\widehat{\F}_h\cdot\nvec\coloneqq 0$.
In particular, when $\varepsilon = \beta_h = 0$, the local energy conservation holds:
\begin{equation*}
    \partial_t \int_K \E \dx +  \int_{\partial K} \widehat{\F}_h \cdot \nvec \dS = 0.
\end{equation*}    
\end{theorem}
\begin{proof}
Let $\charfun_K$ be the characteristic function of the element $K$, and take $\w_h = \charfun_K\u_h$, $\eta_h = \charfun_K s_h$, $r_h = \charfun_K\partial_t \rho_h$ and $\z_h = \charfun_K\partial_t \v_h$ in \eqref{eq:DG_scheme}. Then, taking the sum of \eqref{eq:DG_scheme_r} and \eqref{eq:DG_scheme_z} and using the chain rule, we obtain 
\begin{equation}
    (s_h, \partial_t \rho_h)_K + (\u_h, \partial_t \v_h)_K = \left(\frac{ \partial \E}{\partial \rho_h}, \partial_t \rho_h \right)_K + \left( \frac{ \partial \E}{\partial \v_h}, \partial_t \v_h \right)_K = \partial_t\int_K \E \dx. \label{eq:chain_rule_disc}
\end{equation}
Next, we take the sum of \eqref{eq:DG_scheme_w} and \eqref{eq:DG_scheme_eta} and use \eqref{eq:chain_rule_disc} to obtain 
\begin{align}
  &- (\alpha_h\varepsilon H(\v_h) \Div \u_h, \Div \u_h)_K - (\alpha_h\varepsilon\mathbf{H}(\rho_h)\nabla s_h, \nabla s_h)_K \nonumber \\
  &\quad = (\partial_t \v_h, \u_h)_K + (\partial_t \rho_h, s_h)_K - (s_h,\Div \u_h)_K - (\u_h, \Grad s_h)_K \nonumber \\
  &\qquad\quad + \langle \widehat{s}_h, \u_h \cdot \nvec\rangle_{\partial K} + \langle \widehat{\u_h \cdot \nvec}, s_h \rangle_{\partial K} \nonumber 
  \\
  &\quad = \partial_t\int_K \E\dx -\langle \u_h\cdot \nvec, s_h\rangle_{\partial K}+ \langle \widehat{s}_h, \u_h \cdot \nvec\rangle_{\partial K} + \langle \widehat{\u_h \cdot \nvec}, s_h \rangle_{\partial K} \label{eq:energy_proof_b} \\
  &\quad = \partial_t \int_K \E\dx -\langle \u_h\cdot \nvec, s_h\rangle_{\partial K} + \langle \av{s_h} , \u_h \cdot \nvec\rangle_{\partial K} \nonumber \\
  &\qquad\quad +\langle\beta_h\widetilde{H(\v_h)}\jump{\u_h}\cdot{\nvec}, \u_h \cdot \nvec \rangle_{\partial K} + \langle\av{\u_h }\cdot \nvec, s_h \rangle_{\partial K} \nonumber\\
  &\qquad\quad + \langle \beta_h \widetilde{\mathbf{H}(\rho_h)}\jump{s_h}, s_h \rangle_{\partial K}. \nonumber 
\end{align}
On wall facets, the three boundary terms in step~\eqref{eq:energy_proof_b} cancel, since $\widehat{\u_h\cdot\nvec}=0$ and $\widehat{s}_h=s_h$ there, consistently with $\widehat{\F}_h\cdot\nvec=0$ and with the vanishing jumps on boundary facets; hence, it suffices to consider the interior facets of $\partial K$ in the remaining steps.
 Then the proof is concluded using the identity 
    \begin{equation*}
    \begin{split}
        -\langle \u_h\cdot \nvec, s_h\rangle_{\partial K} + \langle \av{s_h} , \u_h \cdot \nvec\rangle_{\partial K}  + \langle\av{\u_h }\cdot \nvec, s_h \rangle_{\partial K} &= \int_{\partial K}\av{\u_h, s_h}_{\times} \cdot \nvec \dS \\
        &= \int_{\partial K}\widehat{\F}_h \cdot \nvec \dS
    \end{split}
    \end{equation*}
and, for the jump terms, the splitting of the traces from inside $K$ into average and jump, $s_h=\av{s_h}+\frac{1}{2}\jump{s_h}$ and $\u_h\cdot\nvec=\av{\u_h}\cdot\nvec+\frac{1}{2}\jump{\u_h}\cdot\nvec$ on $\partial K$, which gives
\begin{equation*}
\begin{aligned}
    &\langle\beta_h\widetilde{H(\v_h)}\jump{\u_h}\cdot{\nvec}, \u_h \cdot \nvec \rangle_{\partial K} + \langle \beta_h \widetilde{\mathbf{H}(\rho_h)}\jump{s_h}, s_h \rangle_{\partial K} \\
    &\qquad = \int_{\partial K}\widehat{\G}_h\cdot\nvec\dS + \frac{1}{2}\int_{\partial K}\beta_h\Big(\widetilde{\mathbf{H}(\rho_h)}\jump{s_h}^2 + \widetilde{H(\v_h)}\big(\jump{\u_h}\cdot \nvec\big)^2\Big)\dS,
\end{aligned}
\end{equation*}
together with the definition \eqref{eq:Pi_loc} of $\Pi_{K,h}$.
Energy conservation when $\varepsilon = \beta_h = 0$ follows trivially.
\end{proof}
\begin{corollary}[Global energy conservation/dissipation] 
\label{cor:energy_glob}
Let $(\v_h,\u_h, \rho_h,s_h)$ be a solution of \eqref{eq:DG_scheme}. Then the total energy satisfies the balance
\begin{equation*}
    \partial_t \int_{\Omega} \E\dx = -\sum_{K\in\mathcal{T}_h} \Pi_{K, h}\leq 0, 
\end{equation*}
and the dissipation vanishes when $\varepsilon = \beta_h = 0$.
\end{corollary}
\begin{proof}
We take the sum of \eqref{eq:energy_balance_loc} over $K\in\mathcal{T}_h$ and use the fact that $(\widehat{\F}_h+\widehat{\G}_h)\cdot \nvec$ is single-valued across interior facets and vanishes on wall facets (for periodic boundary conditions, all facets are interior facets).
\end{proof}
\begin{remark}[Geometric interpretation of the discrete method]
   Actually, system~\eqref{eq:DG_scheme} can be interpreted as a discrete Hamiltonian system with gradient-flow dissipation in the sense of \cite{LombardiPagliantini25}. More precisely, the terms in \eqref{eq:DG_scheme_w} and \eqref{eq:DG_scheme_eta} that do not involve time derivatives can be decomposed into a skew-symmetric contribution, corresponding to the reversible dynamics, and a negative semi-definite contribution, corresponding to the irreversible dissipative dynamics. Since this structure is not needed in the subsequent analysis, we omit the details.
   \hfill $\blacksquare$
\end{remark}
We show now that the proposed DG method conserves $\Curl \v_h$ in a weak sense.
\begin{lemma}[Involution property]
\label{lem:involution}
Let $(\v_h,\u_h, \rho_h,s_h)$ be a solution of \eqref{eq:DG_scheme}. Then $\v_h$ satisfies
\begin{equation*}
    (\partial_t \v_h, \Curl \q)_{\mathcal{T}_h} = 0 \qquad \forall \q \in \mathbf{Q}_h \coloneqq \{ \q \in \mathbf{H}_0(\operatorname{curl}, \Omega) \mid \Curl \q \in \mathbf{P}_{k}(\mathcal{T}_h)\}.
\end{equation*}
\end{lemma}
\begin{proof}
Let $\q \in \mathbf{Q}_h$. We choose $\w_h=\Curl \q$ in \eqref{eq:DG_scheme_w}. This yields
\begin{equation*}
          (\partial_t \v_h, \Curl \q)_{\mathcal{T}_h} - (s_h, \Div \Curl \q)_{\mathcal{T}_h} + \langle \widehat{s}_h, (\Curl \q) \cdot \nvec\rangle_{\partial \mathcal{T}_h} = - b_h(\v_h;\u_h, \Curl \q).
\end{equation*}
Since $\Div\Curl=0$, we immediately obtain
\begin{equation*}
    (s_h, \Div \Curl \q)_{\mathcal{T}_h} = b_h(\v_h; \u_h, \Curl \q) = 0.
\end{equation*}
It remains to consider the interface  
term. Since $\widehat{s}_h$ is single-valued on mesh interfaces, the definition of the jump gives
\begin{equation*}
        \langle \widehat{s}_h, (\Curl \q) \cdot \nvec \rangle_{\partial \mathcal{T}_h} = \frac{1}{2}\langle \widehat{s}_h, \jump{ \Curl \q}\cdot \nvec\rangle_{ \partial \mathcal{T}_h}.
\end{equation*}
Since $\Curl \q \in \mathbf{H}_0(\operatorname{div}, \Omega) \cap\mathbf{P}_k(\mathcal{T}_h)$, it is continuous in the normal direction across mesh interfaces, and this last term vanishes as well.  
\end{proof}
\begin{remark}[Involution preservation in DG methods]
   To the best of the authors' knowledge, the fact that DG discretizations preserve involutions on simplicial meshes was discovered by Fezoui, Lanteri, Lohrengel and  Piperno \cite{DGMaxwell_PipernoEtAl} in the context of the linear Maxwell equations. The case of quadrilateral meshes is more involved and has been discussed by Perrier \cite{Perrier2,Perrier_contraintDG}. For the grad--div and curl--curl operators in first-order form, Ern and Guermond~\cite{ErnGuermond2023} have shown that the DG approximation is involution-preserving and spectrally correct, where the discrete involutions are expressed as discrete orthogonality properties. This weak, i.e. variational, preservation of the involution has to be distinguished from the locally divergence-free DG methods of Cockburn, Li and Shu \cite{CockburnLiShu2004}, from the exactly div-curl-grad compatible DG schemes on unstructured simplex meshes of Abgrall, Dumbser and Maire \cite{CompatibleDG1}, and from the curl-preserving finite volume schemes of \cite{Perrier3,HTCLagrangeGPR,FourSplitUS}, which satisfy the involutions exactly in a strong sense.
   \hfill $\blacksquare$
\end{remark}

\subsection{Variants}
We discuss an alternative stabilization strategy (Section~\ref{sec:variant_stab}) and, in the case~$\mathbf{u}=\mathbf{v}$, an alternative discretization in space (Section~\ref{sec:variant_space}).
\subsubsection{A different stabilization strategy}\label{sec:variant_stab}
Following the approach of Lu, Liu, and Shu \cite{OE_conservationlaws, OE_systems}, the bilinear forms $b_h(\cdot;\cdot,\cdot)$ and $c_h(\cdot;\cdot,\cdot)$ may be replaced by
\begin{align*}
\widetilde{b}_h(\v_h;\u_h,\w_h)
&\coloneqq
\sum_{K\in\mathcal{T}_h}\sum_{\ell=0}^{k}
\frac{\sigma_K^\ell}{h_K}
\bigl(\u_h-\pi_h^{\ell-1}\u_h,\w_h\bigr)_K,\\
\widetilde{c}_h(\rho_h;s_h,\eta_h)
&\coloneqq
\sum_{K\in\mathcal{T}_h}\sum_{\ell=0}^{k}
\frac{\sigma_K^\ell}{h_K}
\bigl(s_h-\pi_h^{\ell-1}s_h,\eta_h\bigr)_K,
\end{align*}
respectively. Here $\sigma_K^\ell\geq0$ are stabilization coefficients, possibly depending on the discrete solution, and~$\pi_h^\ell$ denotes the~$L^2$-orthogonal projection onto the space of discontinuous polynomials of degree at most~$\ell$. We adopt the convention $\pi_h^{-1}=\pi_h^0$.

With these substitutions, Lemma~\ref{lem:conservativity}, Theorem~\ref{thm:energy_loc} and Corollary~\ref{cor:energy_glob} remain valid, \emph{mutatis mutandis}, where the viscous contribution to $\Pi_{K,h}$ in~\eqref{eq:Pi_loc} is replaced by the damping contribution to the local energy/entropy dissipation
\begin{equation}
\widetilde{\Pi}_{K,h} \coloneqq \sum_{\ell=0}^{k}\frac{\sigma_K^\ell}{h_K}\Big(\big\lVert \u_h-\pi_h^{\ell-1}\u_h\big\rVert_K^2+\big\lVert s_h-\pi_h^{\ell-1}s_h\big\rVert_K^2\Big)\geq 0,
\label{eq:Pi_damping}
\end{equation}
which follows from $\widetilde{b}_h(\v_h;\u_h,\charfun_K\u_h)+\widetilde{c}_h(\rho_h;s_h,\charfun_K s_h)=\widetilde{\Pi}_{K,h}$, since $\u_h-\pi_h^{\ell-1}\u_h$ and $s_h-\pi_h^{\ell-1}s_h$ are $L^2(K)$-orthogonal to $\pi_h^{\ell-1}\u_h$ and $\pi_h^{\ell-1}s_h$, respectively. However, unlike the stabilization proposed in this work, the operators $\widetilde{b}_h$ and $\widetilde{c}_h$ do not preserve the involution property proved in Lemma~\ref{lem:involution}. Consequently, this alternative stabilization is recommended only in one space dimension, where the curl involution is trivially satisfied, or in applications in which the involution is not relevant.
\begin{remark}[Damping mechanism and entropy stability]
In \cite{OE_entropystable}, the authors modify 
the original damping mechanism proposed in \cite{OE_conservationlaws, OE_systems} 
in order to obtain an entropy-stable scheme: the contributions associated with the higher-order projections $\pi_h^\ell$, $\ell\ge1$, are not retained, as they are not compatible with the discrete entropy inequality.  
In contrast, in the present formulation the damping acts on the dual variables $\u_h$ and $s_h$, which are also the test functions of the energy balance, so that every projection level contributes a nonnegative term to the local energy/entropy dissipation $\widetilde{\Pi}_{K,h}$ in~\eqref{eq:Pi_damping}; hence, the scheme remains entropy stable while retaining all projection levels, thus preserving the full structure of the original damping mechanism.
\hfill $\blacksquare$
\end{remark}

\subsubsection{A variant for the case
\texorpdfstring{$\mathbf{u}=\mathbf{v}$}{u = v}}\label{sec:variant_space}
In this subsection we assume that the energy $\E$ is a quadratic function of $\v$, that is $\E(\v, \rho) = \frac{1}{2}\lvert \v\rvert^2 + g(\rho)$, and therefore $\u = \v$. This hypothesis allows us to use a different discretization in space. The method reads as follows: Find $\v_h(t)\in \Pdiv:= \mathbf{P}_{k+1}(\mathcal{T}_h) \cap \mathbf{H}_0(\operatorname{div}, \Omega)$, $\rho_h, s_h \in P_k(\mathcal{T}_h)$ such that, for all $\w_h \in \Pdiv$ and $\eta_h, r_h \in P_k(\mathcal{T}_h)$,
\begin{subequations}  \label{eq:BDM}
    \begin{align}
        (\partial_t \v_h, \w_h)_{\Omega} - (s_h, \Div \w_h)_{\Omega} & =-b_h(\v_h; \v_h, \w_h), \label{eq:BDM_w}  \\
        (\partial_t \rho_h, \eta_h)_{\Omega} + (\Div \v_h, \eta_h)_{\Omega} + \langle \beta_h \widetilde{\mathbf{H}(\rho_h)}\jump{s_h}, \eta_h\rangle_{\partial \mathcal{T}_h}  &= -c_h(\rho_h; s_h, \eta_h),\\
        (s_h, r_h)_{\Omega} - \left( \frac{ \partial \E}{\partial \rho_h}, r_h\right)_{\Omega} &= 0.
    \end{align}
\end{subequations}
Here, the space $\mathbf{H}_0(\operatorname{div},\Omega)$ imposes the wall boundary condition $\v_h\cdot\nvec=0$ strongly; for periodic boundary conditions, it is replaced by its periodic counterpart.
\begin{remark}[No bubble function required]
For this variant, the bubble function $\alpha_h$ is not required to ensure consistency of~$b_h$, since the space $\mathbf{P}_{k+1}(\mathcal{T}_h) \cap \mathbf{H}_0(\operatorname{div},\Omega)$ is already $\operatorname{div}$-conforming; hence, in \eqref{eq:BDM} and in~\eqref{eq:BDM_h} below, $b_h$ is defined as in~\eqref{eq:bh}, except that the bubble function is omitted: $b_h(\v_h;\v_h,\w_h)=(\varepsilon H(\v_h)\Div\v_h,\Div\w_h)_{\mathcal{T}_h}$, while $c_h$ is defined as in~\eqref{eq:ch}. 
\hfill $\blacksquare$
\end{remark}
Inspecting the proofs of Lemma~\ref{lem:conservativity} and Theorem~\ref{thm:energy_loc}, one may at first glance expect that it is impossible to establish both local conservation and discrete energy/entropy balance, since neither $\e_i \chi_K$ nor $\v_h \chi_K$ are admissible test functions in \eqref{eq:BDM_w}. To overcome this difficulty, we introduce an equivalent hybridized formulation, following  Arnold and Brezzi \cite{ArnoldBrezzi85}, where~$\widehat{s}_h$ is introduced as an additional discretization variable. Specifically, we seek $\v_h(t)\in \mathbf{P}_{k +1}(\mathcal{T}_h)$, $\rho_h, s_h \in P_k(\mathcal{T}_h)$ and $\widehat{s}_h \in \widehat{M}_{k+1}$ such that, for all $\w_h \in \mathbf{P}_{k+1}(\mathcal{T}_h)$, $\eta_h, r_h \in P_k(\mathcal{T}_h)$ and $\widehat{r}_h \in \widehat{M}_{k+1}$,
\begin{subequations}
\label{eq:BDM_h}
    \begin{align}
        (\partial_t \v_h, \w_h)_{\mathcal{T}_h} - (s_h, \Div \w_h)_{\mathcal{T}_h} + \langle \widehat{s}_h, \w_h \cdot \nvec\rangle_{\partial \mathcal{T}_h} & = -b_h(\v_h; \v_h, \w_h),  \\
        (\partial_t \rho_h, \eta_h)_{\mathcal{T}_h} + (\Div \v_h, \eta_h)_{\mathcal{T}_h} + \langle \beta_h \widetilde{\mathbf{H}(\rho_h)}\jump{s_h}, \eta_h\rangle_{\partial \mathcal{T}_h} &=-c_h(\rho_h; s_h, \eta_h),\\
        (s_h, r_h)_{\mathcal{T}_h} - \left( \frac{ \partial \E}{\partial \rho_h}, r_h\right)_{\mathcal{T}_h} &= 0, \\
        \langle \v_h \cdot \nvec, \widehat{r}_h \rangle_{\partial \mathcal{T}_h} &= 0.
    \end{align}
\end{subequations}
Here $\widehat{M}_{k+1} =\{ \widehat{r}_h \in P_{k+1}(\partial \mathcal{T}_h) \mid \jump{\widehat{r}_h} = 0\}$. Since discontinuous polynomials are admissible test functions in \eqref{eq:BDM_h}, the proofs of local conservation, local and global energy/entropy balance, and involution property carry over \emph{mutatis mutandis} from Section~\ref{sec:properties}. In particular, the numerical flux for $\u_h = \v_h$ is now given by
\begin{equation*}
    \widehat{\u_h\cdot \nvec}
\coloneqq
\v_h\cdot \nvec
+
\beta_h \widetilde{\mathbf{H}(\rho_h)}\jump{s_h},
\end{equation*}
and the local energy/entropy balance \eqref{eq:energy_balance_loc} holds with
\begin{align}
\widehat{\F}_h\cdot\nvec &= \widehat{s}_h\,\v_h\cdot\nvec, \qquad \widehat{\G}_h\cdot\nvec = \beta_h\widetilde{\mathbf{H}(\rho_h)}\jump{s_h}\av{s_h}, \label{eq:BDM_flux}\\
\Pi_{K,h} &= (\varepsilon H(\v_h)\Div\v_h,\Div\v_h)_K + (\alpha_h\varepsilon\mathbf{H}(\rho_h)\Grad s_h,\Grad s_h)_K \nonumber\\
&\quad + \frac{1}{2}\int_{\partial K}\beta_h\widetilde{\mathbf{H}(\rho_h)}\jump{s_h}^2\dS, \nonumber 
\end{align}
where $\widehat{s}_h$ is the hybrid variable; both fluxes in \eqref{eq:BDM_flux} are single-valued and vanish on wall facets.
\begin{remark}[Local properties via hybridization]
Hybridization has proven to be a powerful tool for establishing \emph{local} properties of finite element methods whose shape functions have support extending across multiple elements; see, for example~\cite{BKSternMaxwell, BKSternYangMills, McLachlanSternMS, SternZampa26}.
\hfill $\blacksquare$
\end{remark}

\section{Fully discrete space--time method}\label{sec:discrxt}
To complete the definition of the fully discrete method, we now turn to the time discretization. Observe first that \eqref{eq:DG_scheme} defines a differential--algebraic system of equations (DAE). However, by eliminating the variables $s_h$ and $\u_h$ through \eqref{eq:DG_scheme_r} and \eqref{eq:DG_scheme_z}, respectively, the system can be recast as a system of ordinary differential equations (ODEs), to which standard time-integration techniques may be applied. For example, for the sake of efficiency, one may employ an explicit Runge--Kutta scheme. Indeed, an explicit time integrator requires only the inversion of local mass matrices at each time step and, in particular, avoids the need for nonlinear solvers. By contrast, formulations based solely on energy/entropy variables \cite{HughesFrancaMallet_86I, HiltebrandMishra_spacetimeDG} generally require the solution of local nonlinear problems, if local conservation is to be retained, even when coupled with explicit time-stepping methods. Moreover, Runge--Kutta methods preserve both the conservation properties and the curl involution. The global energy/entropy balance can be restored by the relaxation Runge--Kutta methods \cite{Ketcheson2019,RanochaEtAl2020}, which rescale each time step by a scalar relaxation parameter. A local variant of this approach, which requires the solution of one scalar equation per element, yields fully discrete local entropy inequalities \cite{RanochaDalcinParsani2020}. However, as noted in \cite{ChenShuReview}, whether Runge--Kutta methods without relaxation also preserve \emph{local} energy/entropy balances remains an open question. Fully discrete entropy stability has also been achieved with entropy-stable space--time DG schemes based on summation-by-parts operators \cite{FriedrichEtAl2019} and with high order ADER-DG schemes combined with a relaxation approach \cite{GaburroEtAl2023}.

To ensure that energy/entropy preservation holds at the fully discrete level, we employ a continuous Galerkin--discontinuous Galerkin (CG--DG) Petrov--Galerkin discretization in time with continuous piecewise-polynomial trial functions of degree at most~$m+1$ and discontinuous piecewise-polynomial test functions of degree at most~$m$. 
The resulting scheme is fully implicit: on each space--time slab all elements are coupled through the numerical fluxes, so that a global nonlinear system has to be solved in every time step.

Before introducing the method, we make a slight change in notation. For a tensor-product space--time mesh, we denote the time-step size by $\htime$, the spatial mesh size by~$\hs$, and the space--time mesh size by~$\hst$. Correspondingly, these subscripts will be used to distinguish quantities associated with the temporal, spatial, and space--time discretizations, respectively. 

Let~$\mathcal{T}_{\hs}$ be a mesh of~$\Omega$ as in Section~\ref{sec:meshes}. 
Furthermore, let~$0=t^0<t^1<t^2<\ldots$ define a temporal partition with time steps~$t^{n+1}-t^n\le\htime$.
We define the associated tensor-product space--time mesh by~$\mathcal{T}_{\hst}:=\{K\times (t^n,t^{n+1}):\, K\in \mathcal{T}_{\hs},\ n=0,1,2,\ldots\,\}$.
The fully discrete method reads as follows: At each time step~$n$, $n=0,1,2,\ldots\,$, given $\v_{\hs}^n\in \P_k(\mathcal{T}_{\hs})$ and $\rho_{\hs}^n\in P_k(\mathcal{T}_{\hs})$, find 
\begin{equation*}
\begin{alignedat}{2}
        &\v_{\hst} \in \P_{k}(\mathcal{T}_{\hs})\otimes P_{m+1}([t^n, t^{n+1}]), \qquad &&\rho_{\hst} \in P_k(\mathcal{T}_{\hs}) \otimes P_{m+1}([t^n, t^{n+1}]), \\
        &\u_{\hst} \in \P_{k}(\mathcal{T}_{\hs}) \otimes P_m([t^n, t^{n+1}]), \qquad &&s_{\hst}  \in P_k(\mathcal{T}_{\hs}) \otimes P_m([t^n, t^{n+1}]),
\end{alignedat}
\end{equation*}
such that $\v_{\hst}(\cdot, t^n) = \v_{\hs}^n$, $\rho_{\hst}(\cdot, t^n) = \rho_{\hs}^n$ and, for all 
\begin{equation*}
\begin{alignedat}{2}
&\w_{\hs}, \z_{\hs} \in \P_k(\mathcal{T}_{\hs}), \qquad & \eta_{\hs}, r_{\hs} \in P_k(\mathcal{T}_{\hs}),\\
&\mu_{\htime}, \lambda_{\htime}, \zeta_{\htime}, \theta_{\htime} \in P_{m} 
([t^n, t^{n+1}])
\end{alignedat}
\end{equation*}
the following equations hold:
\begin{subequations}
\label{eq:spacetime}
    \begin{align}
        &\int_{t^n}^{t^{n+1}} \Big[(\partial_t \v_{\hst}, \w_{\hs})_{\mathcal{T}_{\hs}} - (s_{\hst}, \Div \w_{\hs})_{\mathcal{T}_{\hs}} \nonumber\\
        &\qquad\quad + \langle \widehat{s}_{\hst}, \w_{\hs}\cdot \nvec \rangle_{\partial \mathcal{T}_{\hs}} + b_h(\v_{\hs}^n; \u_{\hst}, \w_{\hs})\Big]\mu_{\htime}\dt = 0, \label{eq:spacetime_a}\\
        &\int_{t^n}^{t^{n+1}}\Big[(\partial_t \rho_{\hst}, \eta_{\hs})_{\mathcal{T}_{\hs}} - (\u_{\hst}, \Grad \eta_{\hs})_{\mathcal{T}_{\hs}} \nonumber\\
        &\qquad\quad + \langle \widehat{\u_{\hst} \cdot \nvec }, \eta_{\hs} \rangle_{\partial \mathcal{T}_{\hs}} + c_h(\rho_{\hs}^n; s_{\hst}, \eta_{\hs})\Big]\lambda_{\htime}\dt = 0, \label{eq:spacetime_b}\\
        &\int_{t^n}^{t^{n+1}}\left[(s_{\hst}, r_{\hs})_{\mathcal{T}_{\hs}} - \left( \frac{\partial \E}{\partial \rho_{\hst}}, r_{\hs} \right)_{\mathcal{T}_{\hs}} \right]\zeta_{\htime}\dt = 0, \label{eq:spacetime_c}\\
        &\int_{t^n}^{t^{n+1}} \left[(\u_{\hst}, \z_{\hs})_{\mathcal{T}_{\hs}} - \left( \frac{ \partial \E}{\partial \v_{\hst}}, \z_{\hs} \right)_{\mathcal{T}_{\hs}} \right]\theta_{\htime} \dt = 0. \label{eq:spacetime_d}
    \end{align}
\end{subequations}
Finally, set $\v_{\hs}^{n+1} \coloneqq \v_{\hst}(\cdot, t^{n+1})$ and $\rho_{\hs}^{n+1} \coloneqq \rho_{\hst}(\cdot, t^{n+1})$.

\begin{remark}[Nonlinear structure of the scheme]\label{rem:nonlinear_structure}
 Nonlinearities appear in formulation~\eqref{eq:spacetime} \emph{i)}~in equations~\eqref{eq:spacetime_c} and~\eqref{eq:spacetime_d}, 
which are block diagonal, \emph{ii)} in~the forms~$b_h$ and~$c_h$, which are evaluated at time~$t^n$, and \emph{iii)} in~the coefficients $\beta_h$, $\widetilde{\mathbf{H}(\rho_h)}$, and $\widetilde{H(\v_h)}$ of the jump terms in the numerical fluxes~\eqref{eq:fluxes} (the latter coupling neighboring mesh elements), which we evaluate at time~$t^n$ as well. Since the terms in~\emph{ii)} and~\emph{iii)} are introduced solely for regularization and stabilization, their explicit evaluation at~$t^n$ does not compromise the order of accuracy observed in the coupled space--time refinement of Section~\ref{sec:convergence_test}. In this way, nonsmooth nonlinearities are avoided and the block-diagonal structure of the remaining nonlinear terms allows the nonlinear solver, which is required by implicit time integration, to be implemented efficiently. Structure-preserving DG methods based on entropy variables have also been developed for nonlinear parabolic systems, with nonlinearities treated locally within each element~\cite{GJP25}.
\end{remark}

\begin{remark}[$L^2$ projection for the dual variables]
The variables $s_{\hst}$ and $\u_{\hst}$ are the \emph{space--time} $L^2$ projections of $\frac{ \partial \E}{\partial \rho_{\hst}}$ and $\frac{ \partial \E}{\partial \v_{\hst}}$ onto $P_k(\mathcal{T}_{\hs}) \otimes P_m([t^n, t^{n+1}])$ and $\P_k(\mathcal{T}_{\hs})\otimes P_m([t^n, t^{n+1}])$, respectively. The $L^2$ projection of the dual variables is a known device both in space \cite{EggerDissipative,EggerShashkov21,LombardiPagliantini25} and in time \cite{BetschSteinmann2000,QuispelMcLaren2008,Hairer2010,BrugnanoIavernaro2016,BrugnanoIavernaro2018,AndrewsFarrell25}: for $m=0$ and $\varepsilon=\beta_h=0$, the time discretization reduces to the averaged vector field method \cite{QuispelMcLaren2008}, which, for the present nonlinear acoustic system, corresponds to the path integral of the dual variables along a straight segment in the state variables used in \cite{LuccaDumbser2026}, while for general $m$ it is closely related to the energy-preserving variant of collocation methods \cite{Hairer2010}, to the line integral methods of Brugnano and Iavernaro \cite{BrugnanoIavernaro2016,BrugnanoIavernaro2018}, and to energy-conserving time finite elements \cite{BetschSteinmann2000}. The new element of the present method is the combination of this projection with a DG discretization in space, which makes the discrete energy/entropy balance local.
\hfill $\blacksquare$
\end{remark}
\begin{theorem}[Properties of the fully discrete method]
\label{thm:cons_fully_discrete}
Let $(\v_{\hst},\allowbreak\rho_{\hst},\allowbreak\u_{\hst},\allowbreak s_{\hst})$ be a solution of the fully discrete space--time scheme \eqref{eq:spacetime}. Then the following properties hold.
\begin{enumerate}
\item \textbf{Local conservation.}
For every element $K\in\mathcal{T}_{\hs}$,
\begin{align*}
\int_K \v_{\hs}^{n+1}\cdot \e_i\dx-\int_K \v_{\hs}^{n}\cdot \e_i\dx
+ \int_{t^n}^{t^{n+1}}\int_{\partial K}
\widehat{s}_{\hst}\,\e_i\cdot\nvec\dS\dt &= 0,\\
\int_K \rho_{\hs}^{n+1}\dx-\int_K \rho_{\hs}^{n}\dx
+ \int_{t^n}^{t^{n+1}}\int_{\partial K}
\widehat{\u_{\hst}\cdot\nvec}\dS\dt &= 0.
\end{align*}
\item \textbf{Local energy/entropy balance.}
For every element $K\in\mathcal{T}_{\hs}$,
\begin{equation*}
\begin{aligned}
&\int_K \E(\v_{\hs}^{n+1},\rho_{\hs}^{{n+1}}
)\dx
- \int_K \E(\v_{\hs}^{n},\rho_{\hs}^{n})\dx \\
  &\qquad
{+} 
\int_{t^n}^{t^{n+1}}\int_{\partial K}
  (\widehat{\F}_{\hst}+\widehat{\G}_{\hst})\cdot\nvec \dS\dt
  = - \int_{t^n}^{t^{n+1}} \Pi_{K,\hst}\dt
  \le 0,
  \end{aligned}
  \end{equation*}
  and the dissipation $\Pi_{K,\hst}$ vanishes if $\varepsilon=\beta_{\hs}=0$. Here $\widehat{\F}_{\hst}$, $\widehat{\G}_{\hst}$ and $\Pi_{K,\hst}$ are defined as in \eqref{eq:Fhat}, \eqref{eq:Ghat} and \eqref{eq:Pi_loc}, with the dual variables evaluated at the space--time solution and all coefficients of the stabilization ($\varepsilon$, $H$, $\mathbf{H}$, $\beta_h$ and the harmonic averages) evaluated as in \eqref{eq:spacetime}.
\item \textbf{Global energy/entropy balance.}
\begin{equation*}
\begin{aligned}
&\int_{\Omega}\E(\v_{\hs}^{n+1},\rho_{\hs}^{n+1})\dx
-\int_{\Omega}\E(\v_{\hs}^{n},\rho_{\hs}^{n})\dx \\
&\qquad =
-\int_{t^n}^{t^{n+1}}
\sum_{K\in\mathcal{T}_{\hs}}
\Pi_{K,\hst}
\dt
\le 0 .
\end{aligned}
\end{equation*}
\item \textbf{Involution.}
For every $\q\in\mathbf{Q}_{\hs}$,
\begin{equation*}
(\v_{\hs}^{n+1}-\v_{\hs}^{n},\Curl\q)_{\mathcal{T}_{\hs}}=0.
\end{equation*}
\end{enumerate}
\end{theorem}

\begin{proof}
Statements~\emph{1.--4.} are proved exactly as Lemma~\ref{lem:conservativity}, Theorem~\ref{thm:energy_loc}, Corollary~\ref{cor:energy_glob}, and Lemma~\ref{lem:involution}, respectively. Since the arguments are identical, we only indicate the corresponding choices of test functions.
To establish conservation on an element $K$, choose in~\eqref{eq:spacetime_a} and~\eqref{eq:spacetime_b}
\begin{equation*}
\w_{\hs} = \charfun_K \e_i,
\qquad
\mu_{\htime} = \frac{1}{\htime},
\qquad
\eta_{\hs} = \charfun_K,
\qquad
\lambda_{\htime} = \frac{1}{\htime}.
\end{equation*}
To prove local energy conservation on \(K\), take in~\eqref{eq:spacetime}
\begin{equation*}
    \begin{aligned}
    &\w_{\hs}\mu_{\htime} = \charfun_K \u_{\hst}, &\qquad &\eta_{\hs}\lambda_{\htime} = \charfun_K s_{\hst},\\
    &r_{\hs}\zeta_{\htime} = \charfun_K \partial_t\rho_{\hst}, &\qquad &\z_{\hs}\theta_{\htime} = \charfun_K \partial_t \v_{\hst}.
    \end{aligned}
\end{equation*}
Global energy conservation then follows by summing the resulting identity over all elements \(K \in \mathcal{T}_{\hs}\).
Finally, to obtain the involution property, choose in~\eqref{eq:spacetime_a}
\begin{equation*}
    \w_{\hs} = \Curl \q,
    \qquad
    \q \in \mathbf{Q}_{\hs}, 
    \qquad
    \mu_{\htime} = \frac{1}{\htime}.
\end{equation*}
With these choices, the arguments of the cited results can be repeated to obtain the stated properties.
\end{proof}

\section{Numerical results}\label{sec:numer}
In this section, we present a numerical assessment of the proposed scheme \eqref{eq:spacetime}, focusing on its convergence properties for smooth solutions, its stability in the presence of shocks, and its discrete conservation properties. We consider the energy density 
\begin{equation*}
    \E(\v,\rho) = \frac{1}{2} \lvert \v \rvert^2 + \frac{1}{\gamma (\gamma + 1)} \rho^{\gamma + 1},
\end{equation*}
where $\gamma \ge 1$ is an  exponent that varies depending on the test case. The domain of physically admissible states is defined by the open convex set
\begin{equation*}
    \mathcal{A} \coloneqq \{ (\v,\rho) \in \mathbb{R}^d \times \mathbb{R} \mid \rho > 0 \}.
\end{equation*}
We note that the Hessian of $\E$ is positive definite on $\mathcal{A}$. Although the scheme is not provably positivity-preserving for  $\rho$, we observe that $\rho$ remains positive in all our numerical experiments.

We evaluate two variants of the proposed framework:
\begin{enumerate}
    \item \textbf{Conservative scheme:} We set $\beta_h = \varepsilon = 0$. By Theorem~\ref{thm:cons_fully_discrete}, this variant satisfies local and global discrete energy/entropy conservation.
    \item \textbf{Dissipative scheme:} We set $\beta_h = \frac{1}{2}$ and define $\varepsilon$ elementwise on the mesh~$\mathcal{T}_{\hs}$ as
    \begin{equation*}
        \varepsilon|_K = C_{\varepsilon}\, h_K \lVert \mathcal{L}_{\hs}^k(\jump{\rho_{\hs}^n})\rVert_{L^2(K)},
    \end{equation*}
    where $h_K$ denotes the local size of the element~$K$, $\mathcal{L}_{\hs}^k$ denotes the lifting operator in~\cite[Chapter 38.4.1]{ErnGuermond-II:2020}, and $C_{\varepsilon}>0$ is a constant; we set $C_{\varepsilon} = 2$ in this work.
This artificial viscosity parameter serves as an effective shock-capturing mechanism. Moreover, because $\mathcal{L}_{\hs}^k(\jump{\rho - \pi_h^k\rho}) \lesssim {\hs}^k$ for sufficiently smooth $\rho$, high-order convergence is preserved in the case of smooth solutions.
\end{enumerate}

While our method does not require a CFL condition for energy/entropy stability (see Theorem~\ref{thm:cons_fully_discrete}), in practice, the time step must be sufficiently small to ensure convergence of the nonlinear iterative solver 
and preserve the positivity of the density. Then, the time step $\htime$ is chosen according to the condition
\begin{equation*}
    \htime = \Ctime \hs,
\end{equation*}
with a constant $\Ctime > 0$. The parameters for each test are detailed in Table~\ref{tab:par}.
Unstructured triangular meshes in space and periodic boundary conditions are used in all tests. The nonlinear system is solved using Newton's method, while each resulting linear system is solved with GMRES and a block-Jacobi preconditioner.

The scheme has been implemented using the open-source finite element library NGSolve \cite{ngSolve}; the source code is publicly available at \url{https://github.com/EnricoZampa/StructurePreservingDG}.
\begin{table}
\centering
\begin{tabular}{c|cccccccc}
    Section & $\Omega$ & $\gamma$ & $\hs$ & $\Ctime$ & $\tend$ &$k$ & $m$ \\
    \hline
    \textsection~\ref{sec:smooth_test} & $[-0.5,0.5]^2$ & $1.4$ & $\frac{1}{40}$& $0.1$ & $0.25$& $2$ & $1$ \\[0.1cm]
    \textsection~\ref{sec:convergence_test} & $[0,2\pi]^2$ & $3$ & $-$ & $\frac{1}{4 \pi}$ & $0.5$& $-$& $-$ \\[0.1cm]
    \textsection~\ref{sec:riemann_test} & $[-1,1]^2$ & $2$ & $\frac{1}{120}$ & $0.1$ & $0.3$ & $1$ & $0$
     \\
\end{tabular}
\caption{Parameters for the numerical tests (``$-$'' denotes a variable parameter).}
\label{tab:par}
\end{table}

\subsection{Conservation properties}
\label{sec:smooth_test}
We begin by verifying the local and global conservation properties of the conservative scheme. This test was proposed in~\cite{LuccaDumbser2026}, with the initial conditions given by 
\begin{equation*}
    \rho(x, y, 0) = 4 + e^{-\frac{1}{2} \left(\frac{x^2 + y^2}{0.05^ 2} \right)}, \qquad \v(x,y,0) = \mathbf{0}.
\end{equation*}
We evaluate:
\begin{enumerate}
    \item The global energy conservation error at each time step:
    \begin{equation*}
     \left\lvert \int_{\Omega}\E(\v^{n+1}_{\hs},\rho^{n+1}_{\hs})\dx - \int_{\Omega}\E(\v_{\hs}^0,\rho^0_{\hs})\dx \right\rvert.
     \end{equation*}
     \item The local energy conservation error on an arbitrary element $K$:
     \begin{equation*}
         \left \lvert \int_{K}\E(\v_{\hs}^{n+1},\rho_{\hs}^{n+1})\dx -\int_K \E(\v_{\hs}^n,\rho_{\hs}^n)\dx + \int_{t^n}^{t^{n+1}}\int_{\partial K}\widehat{\F}_{\hst}\cdot \nvec \dS \dt \right \rvert.
     \end{equation*}
     \item The weak two-dimensional curl of $\v_{\hs}^n$, which is defined as the solution~$\phi_{\hs}^{n}\in P_{k+1}(\mathcal{T}_{\hs})\cap H^1_0(\Omega)$ of 
     \begin{equation*}
         (\phi_{\hs}^n, \psi_{\hs}) = (\v_{\hs}^{n}, \nabla^{\perp}\psi_{\hs}), \qquad \forall \psi_{\hs} \in P_{k+1}(\mathcal{T}_{\hs})\cap H^1_0(\Omega),
     \end{equation*}
     with $\nabla^{\perp}\psi_{\hs} = (\partial_2\psi_{\hs},-\partial_1 \psi_{\hs})$.
\end{enumerate}
The density $\rho$ at the final time is plotted in Figure \ref{fig:rho_tend}, while the conservation properties are plotted in Figure \ref{fig:cons_smooth}.
As expected, the conservation properties hold up to machine precision.
\begin{figure}[htbp]
\centering
\includegraphics[width = \linewidth]{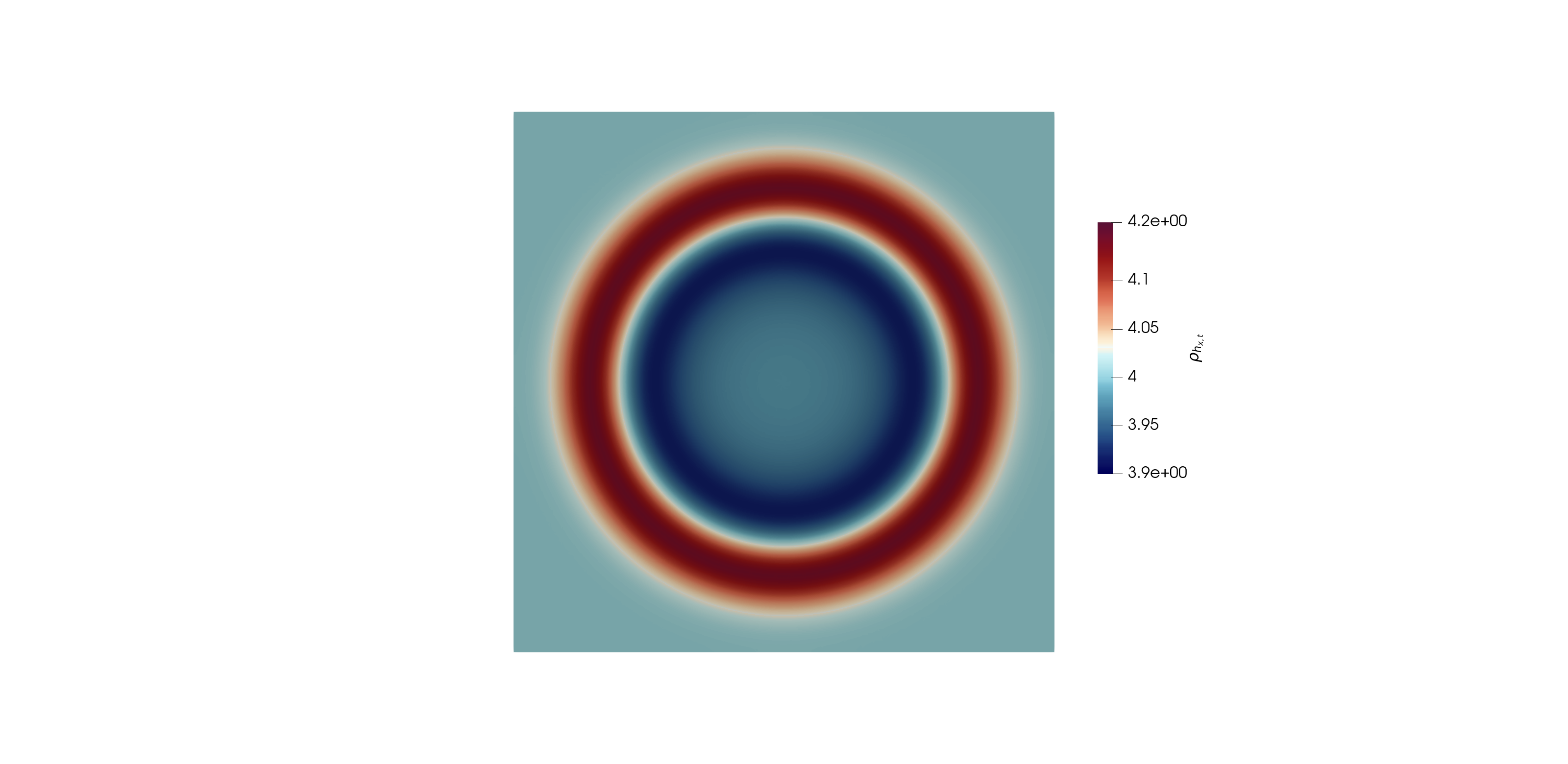}
\caption{$\rho_{\hst}$ at the final time $\tend = 0.25$.}
\label{fig:rho_tend}
\end{figure}
\begin{figure}[htbp]
    \centering
    \begin{subfigure}[t]{0.3\textwidth}
        \centering     \makebox[\linewidth][l]{\includegraphics[width=1.034\linewidth]{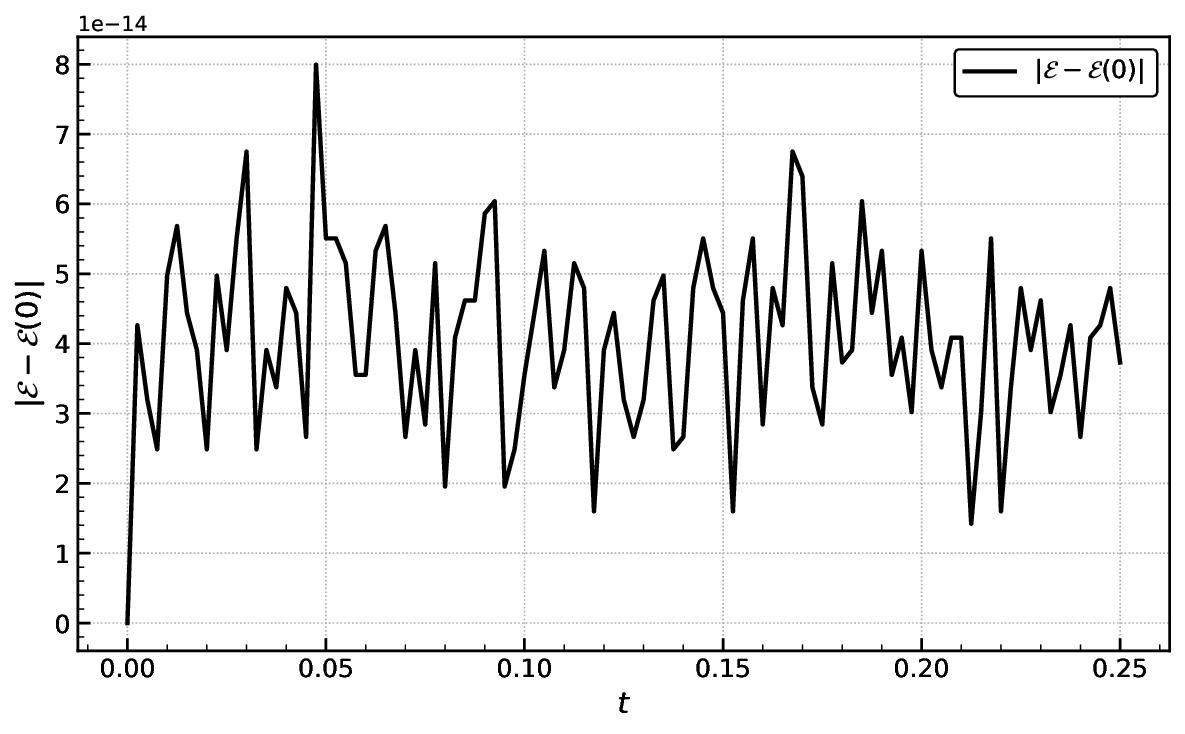}}
        \caption{Global energy conservation error.}
        \label{fig:sub1}
    \end{subfigure}
    \hfill
    \begin{subfigure}[t]{0.3\textwidth}
        \centering
\includegraphics[width=\linewidth]{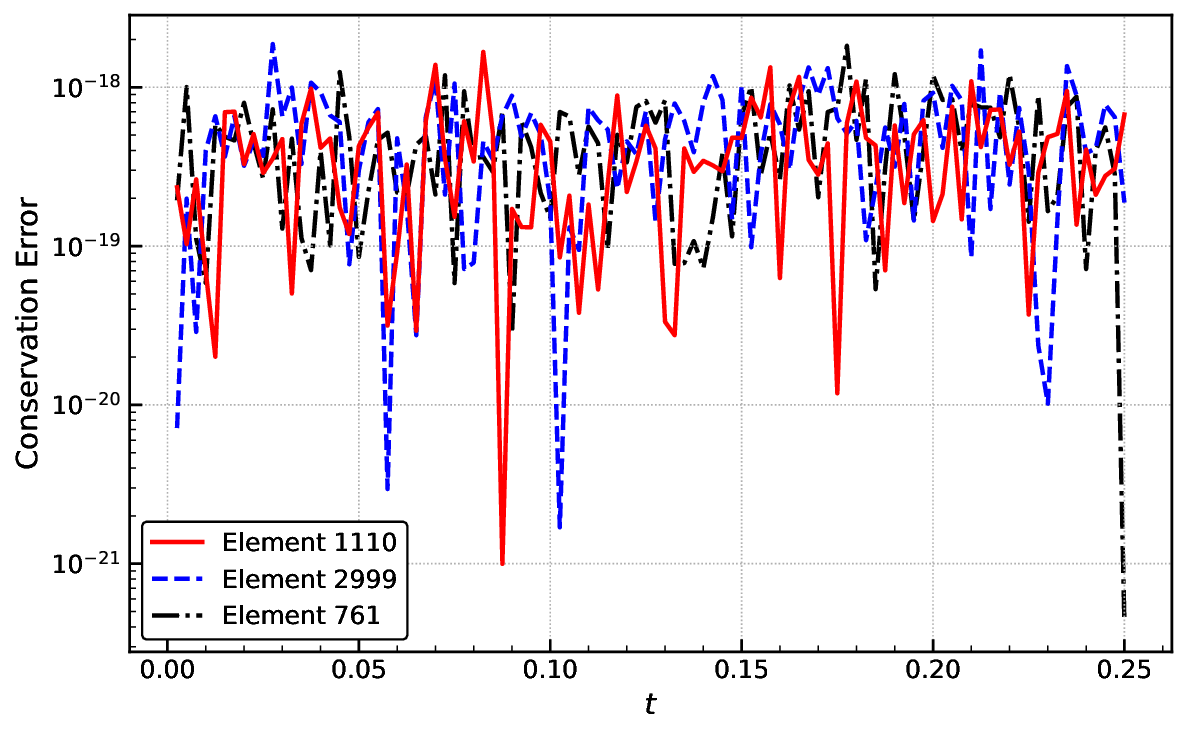}
        \caption{Local energy conservation error on three arbitrarily chosen elements.}
        \label{fig:sub2}
    \end{subfigure}
    \hfill
    \begin{subfigure}[t]{0.3\textwidth}
        \centering
        \includegraphics[width=\linewidth]{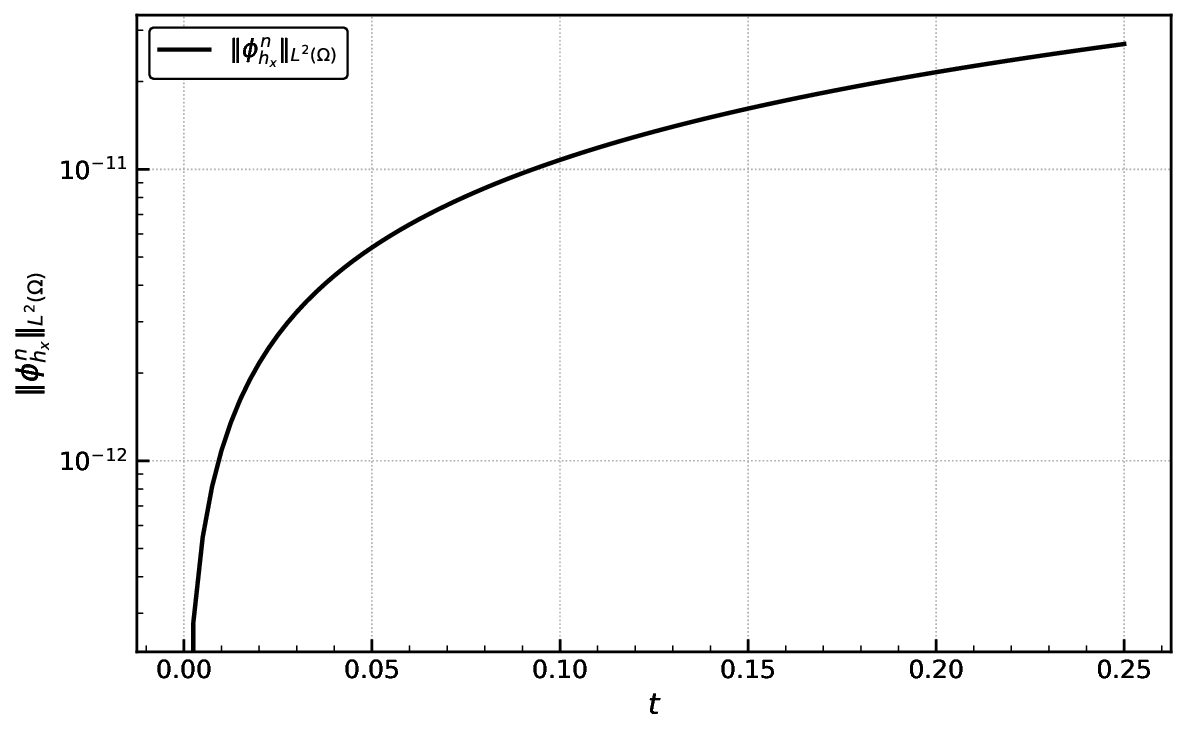}
        \caption{Weak curl of $\v_{\hs}^n$.}
        \label{fig:sub3}
    \end{subfigure}

    \caption{Conservation properties for the test in Section \ref{sec:smooth_test}.}
\label{fig:cons_smooth}
\end{figure}

\subsection{Convergence rates}\label{sec:convergence_test}
For this test case, the exact solution is given by
\begin{equation*}
    \rho(x, y, t) = R(x + y, t), \qquad \v(x, y, t) = \frac{1}{2\sqrt{2}}\rho(x,y,t)^2\begin{pmatrix}1\\1\end{pmatrix},
\end{equation*}
where $R(\xi, t)$ satisfies the Burgers equation
\begin{equation*}
    \partial_t R + \sqrt{2} R \partial_{\xi} R = 0,
\end{equation*}
subject to the initial condition $R(\xi, 0) = 1 + 0.3\sin\xi$. A direct calculation confirms that $(\v,\rho)$ solves \eqref{eq:nonlinear_acoustics} with $\gamma = 3$. The solution remains smooth over short time intervals. We apply both the conservative scheme ($\beta_h = 0, \varepsilon = 0$) and the dissipative scheme. The resulting $L^2$ errors and convergence rates are summarized in Table~\ref{tab:convergence_nonlinear_comparison}. For the conservative scheme, we observe convergence rates of~$k+1$ for~$\rho$ and~$k$ for~$\v$. In contrast, the dissipative scheme achieves the optimal convergence rate of~$k+1$ for both variables. This highlights that the stabilization improves the convergence rate of~$\v$, while the artificial viscosity remains sufficiently small to avoid spoiling the optimal rate.

\begin{table}[htpb]
\centering
\caption{Experimental convergence rates for $\rho$ and $\v$ at the final time $\tend = 0.5$ using the conservative and the dissipative method. Here $n_{\mathrm{el}}$ is such that $\hs = \frac{2\pi}{n_{\mathrm{el}}}$.}
\label{tab:convergence_nonlinear_comparison}
\resizebox{\textwidth}{!}{
\begin{tabular}{cc|cccc|cccc}
\hline
\multirow{2}{*}{$k = m+1$} & \multirow{2}{*}{$n_{\mathrm{el}}$} & \multicolumn{4}{c|}{\textbf{Conservative}} & \multicolumn{4}{c}{\textbf{Dissipative}} \\
\cline{3-10}
 & & $L^2(\rho)$ \textbf{error} & \textbf{Rate} & $L^2(\v)$ \textbf{error} & \textbf{Rate} & $L^2(\rho)$ \textbf{error} & \textbf{Rate} & $L^2(\v)$ \textbf{error} & \textbf{Rate} \\
\hline
1 & 10 & 4.39e-02 & - & 1.42e-01 & - & 3.58e-02 & - & 4.86e-02 & - \\
  & 20 & 9.50e-03 & 2.21 & 6.62e-02 & 1.10 & 7.94e-03 & 2.17 & 1.12e-02 & 2.11 \\
  & 40 & 2.62e-03 & 1.86 & 3.40e-02 & 0.96 & 2.05e-03 & 1.95 & 2.95e-03 & 1.93 \\
\hline
2 & 6  & 2.02e-02 & - & 3.50e-02 & - & 2.30e-02 & - & 2.72e-02 & - \\
  & 12 & 2.11e-03 & 3.26 & 7.37e-03 & 2.25 & 2.03e-03 & 3.50 & 2.91e-03 & 3.23 \\
  & 24 & 2.36e-04 & 3.16 & 1.77e-03 & 2.05 & 2.21e-04 & 3.20 & 3.38e-04 & 3.10 \\
\hline
3 & 4  & 1.50e-02 & - & 2.09e-02 & - & 3.96e-02 & - & 3.61e-02 & - \\
  & 8  & 9.35e-04 & 4.00 & 3.39e-03 & 2.62 & 1.24e-03 & 4.99 & 1.43e-03 & 4.66 \\
  & 16 & 5.91e-05 & 3.98 & 4.08e-04 & 3.06 & 4.69e-05 & 4.73 & 6.75e-05 & 4.40 \\
\hline
\end{tabular}
}
\end{table}

\subsection{Circular explosion}\label{sec:riemann_test}
Finally, to show the robustness of the dissipative method in the presence of shock waves, we consider a one-dimensional Riemann problem in the radial direction. The initial conditions are given by 
\begin{equation}
\rho(x, y, 0) = \begin{cases} 2 \qquad &\text{ if $\sqrt{x^2 + y^2} < 0.6$}, \\
0.1 \qquad &\text{ otherwise,}
\end{cases}\qquad \v(x,y,0) = \mathbf{0}.
\end{equation}
The discontinuous initial data are approximated by their cell averages on~$\mathcal{T}_{\hs}$.
The discrete $\rho_{\hst}$ at the final time $\tend = 0.3$ is displayed in Figure \ref{fig:RP_cut}, together with the reference solution obtained by solving the corresponding one-dimensional Riemann problem
with geometric reaction source terms using a MUSCL--Hancock method on a very fine mesh, see e.g. \cite{toro-book}:  
we see that the shock is correctly captured and the oscillations are essentially suppressed. Moreover, from Figures \ref{fig:RP_energy} and \ref{fig:RP_curl} respectively, we observe that the global energy/entropy is correctly dissipated and the weak curl of $\v_{\hst}$ remains close to machine precision.
\begin{figure}[htbp]
    \centering
    \begin{subfigure}[t]{0.3\textwidth}
        \centering     \includegraphics[width=\linewidth]{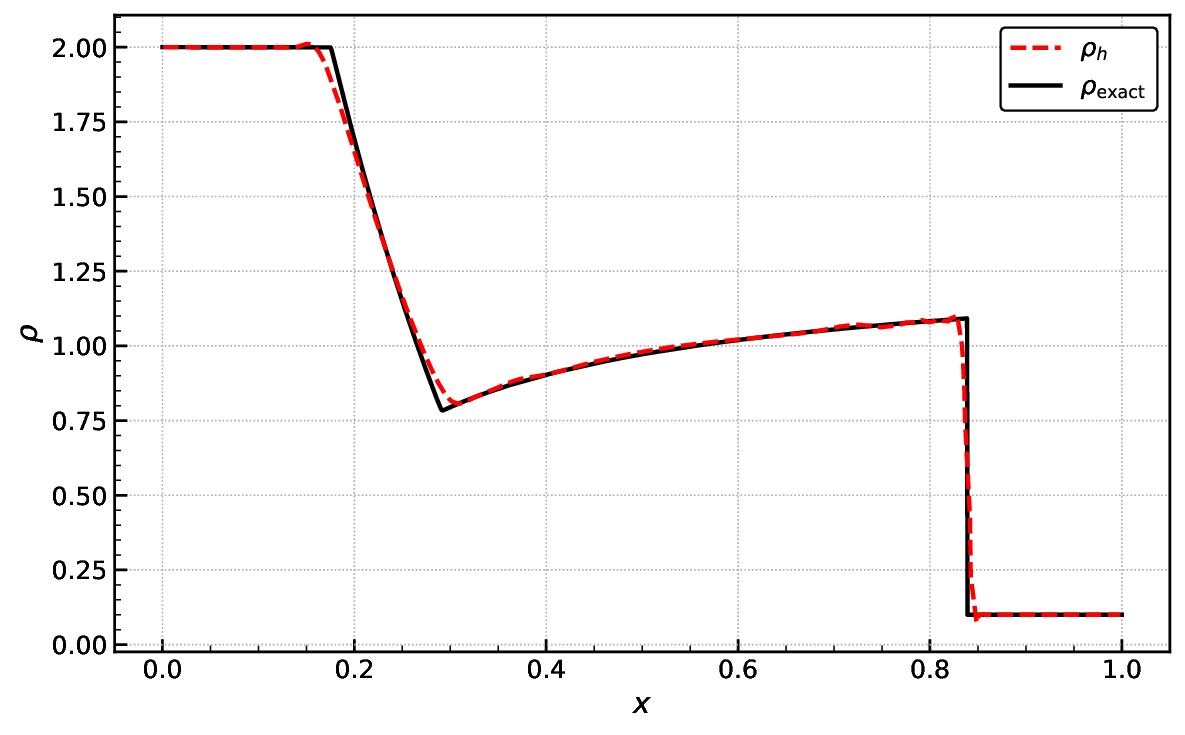}
        \caption{Cut of $\rho_{\hst}$ along $y = 0$ and reference solution.}
        \label{fig:RP_cut}
    \end{subfigure}
    \hfill
    \begin{subfigure}[t]{0.3\textwidth}
        \centering
\includegraphics[width=\linewidth]{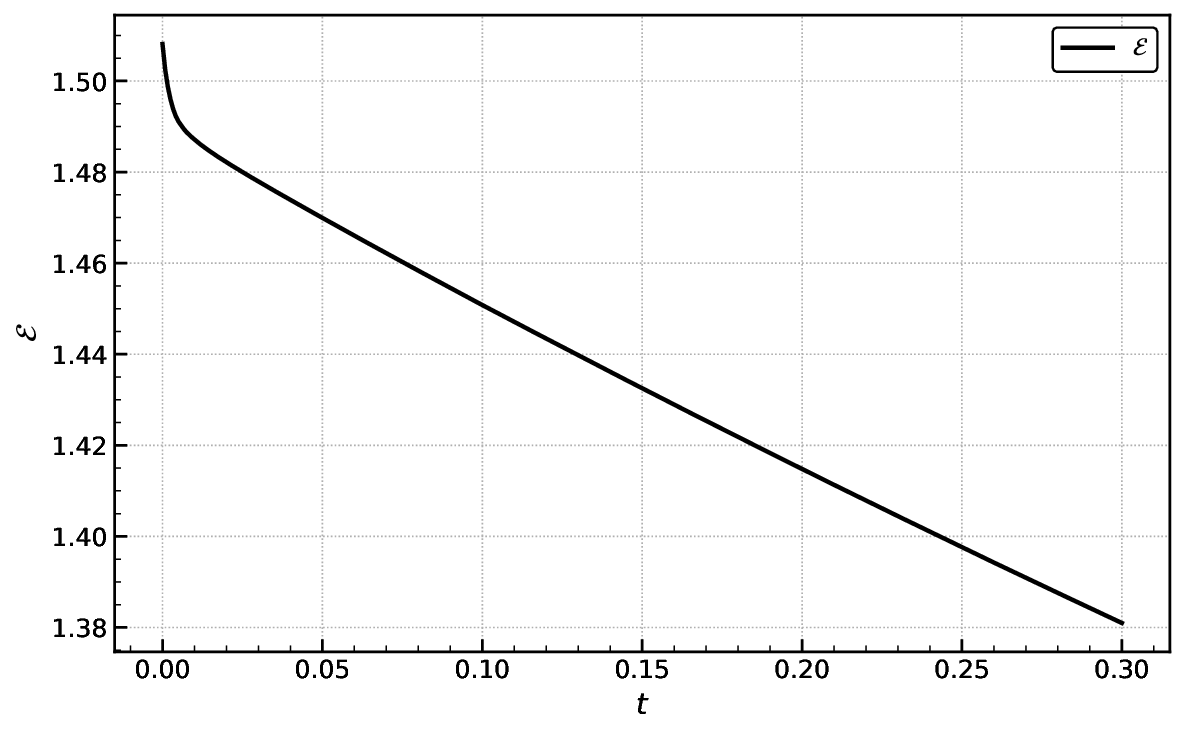}
        \caption{Global energy/entropy.}
        \label{fig:RP_energy}
    \end{subfigure}
    \hfill
    \begin{subfigure}[t]{0.3\textwidth}
        \centering
        \includegraphics[width=\linewidth]{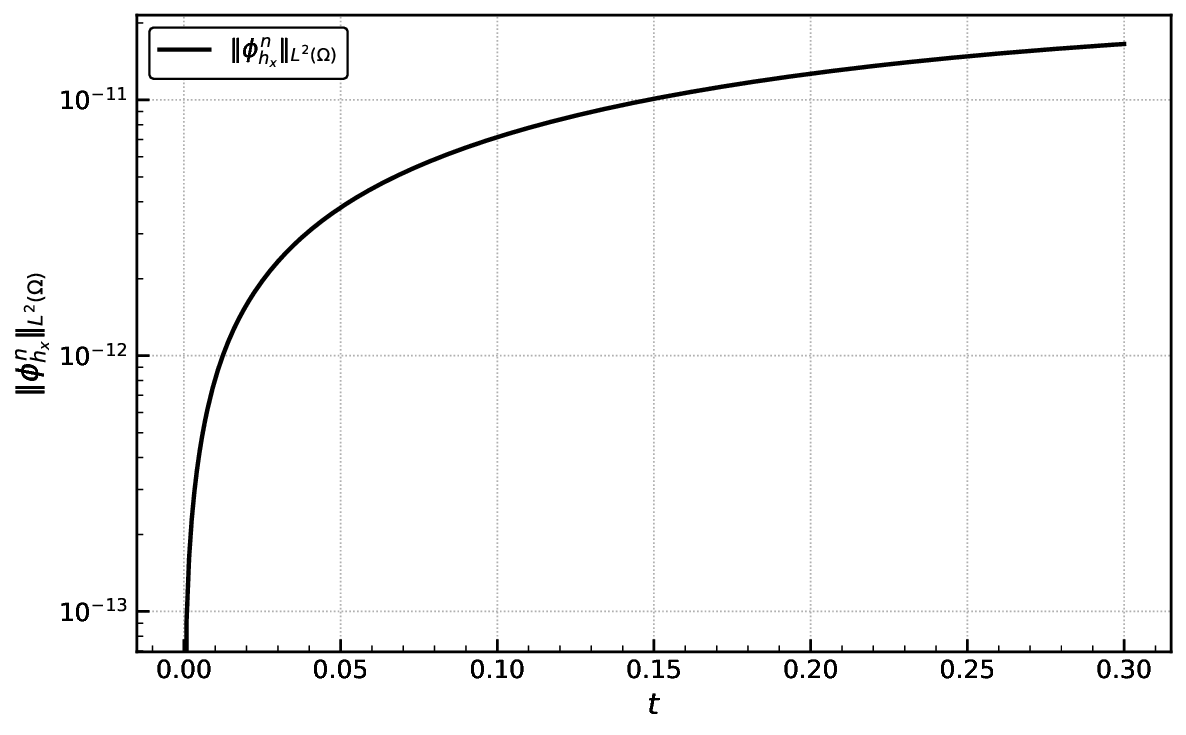}
        \caption{Weak curl of $\v_{\hs}^n$.}
        \label{fig:RP_curl}
    \end{subfigure}

    \caption{Two-dimensional circular explosion problem.}
    \label{fig:RP}
\end{figure}

\section{Conclusions}\label{sec:conclusions}
We introduced a new class of high-order accurate space--time schemes for nonlinear wave systems on unstructured meshes, using a DG discretization in space and a continuous--discontinuous Petrov--Galerkin discretization in time. The key idea consists in representing the state variables and the dual variables in the same DG space, where the dual variables are obtained as $L^2$ projections of the derivatives of the energy with respect to the conserved variables and are subsequently used as test functions in the energy/entropy balance. 
Our new schemes are locally conservative, satisfy a local energy/entropy balance and preserve the curl involution of the velocity field in a weak sense. 
Thanks to a suitable viscous regularization, they preserve the curl involution weakly also in the presence of discontinuities.

Future work will concern nonlinear Maxwell systems and more complex hyperbolic systems in which the fluxes are nonlinear functions of the dual variables, such as the compressible Euler equations. 

\section*{Acknowledgments}

M.D. was funded by the Fondazione Caritro via the project SOPHOS and by the European Research Council (ERC) under the European Union's Horizon Europe research and innovation programme via the project SOPHOS, grant agreement no.~ERC-ADG-2025-101265878-SOPHOS. 

M.D. and E.Z. are members of the Gruppo Nazionale per il Calcolo Scientifico dell'Istituto Nazionale di Alta Matematica (GNCS-INdAM).

The authors gratefully acknowledge the open-source finite element library NGSolve of J.~Sch{\"o}berl \textit{et~al.}~\cite{ngSolve}, available at \url{https://ngsolve.org}, which has been used for the implementation of the proposed schemes. We are furthermore very grateful to Joachim Sch{\"o}berl for all the very inspiring and highly productive scientific discussions.

\bibliographystyle{plain}
\bibliography{bibliography}

@book {ErnGuermond-II:2020,
    AUTHOR = {Ern, A. and Guermond, J.-L.},
     TITLE = {Finite elements {II}---{G}alerkin approximation, elliptic and
              mixed {PDE}s},
    SERIES = {Texts Appl. Math.},
    VOLUME = {73},
 PUBLISHER = {Springer, Cham},
      YEAR = {2021},
     PAGES = {ix+492},
      ISBN = {978-3-030-56922-8; 978-3-030-56923-5},
   MRCLASS = {65-01},
  MRNUMBER = {4269305},
}

@article {MaoXi25,
    AUTHOR = {Mao, S. and Xi, R.},
     TITLE = {An incompressibility, {${\rm div}\,B = 0$} preserving, current
              density, helicity, energy-conserving finite element method for
              incompressible {MHD} systems},
   JOURNAL = {J. Comput. Phys.},
  FJOURNAL = {Journal of Computational Physics},
    VOLUME = {538},
      YEAR = {2025},
     PAGES = {114130},
      ISSN = {0021-9991,1090-2716},
   MRCLASS = {65M60 (65F08 65M06 65M12 76D05 76W05)},
  MRNUMBER = {4917731},
MRREVIEWER = {Xiaodi\ Zhang},
}

@article {HuLeeXu21,
    AUTHOR = {Hu, K. and Lee, Y.-J. and Xu, J.},
     TITLE = {Helicity-conservative finite element discretization for
              incompressible {MHD} systems},
   JOURNAL = {J. Comput. Phys.},
  FJOURNAL = {Journal of Computational Physics},
    VOLUME = {436},
      YEAR = {2021},
     PAGES = {110284},
      ISSN = {0021-9991,1090-2716},
   MRCLASS = {65M60 (76W05)},
  MRNUMBER = {4236017},
}

@article {GawlikGB22,
    AUTHOR = {Gawlik, E. S. and Gay-Balmaz, F.},
     TITLE = {A finite element method for {MHD} that preserves energy,
              cross-helicity, magnetic helicity, incompressibility, and
              {${\rm div}\, B = 0$}},
   JOURNAL = {J. Comput. Phys.},
  FJOURNAL = {Journal of Computational Physics},
    VOLUME = {450},
      YEAR = {2022},
     PAGES = {110847},
      ISSN = {0021-9991,1090-2716},
   MRCLASS = {65M60 (76W05)},
  MRNUMBER = {4346694},
}

@article {EggerShashkov21,
    AUTHOR = {Egger, H. and Habrich, O. and Shashkov, V.},
     TITLE = {On the energy stable approximation of {H}amiltonian and
              gradient systems},
   JOURNAL = {Comput. Methods Appl. Math.},
  FJOURNAL = {Computational Methods in Applied Mathematics},
    VOLUME = {21},
      YEAR = {2021},
    NUMBER = {2},
     PAGES = {335--349},
      ISSN = {1609-4840,1609-9389},
   MRCLASS = {37K06 (37L65 47J35 65J08)},
  MRNUMBER = {4235809},
}

@Misc{LombardiPagliantini25,
  author = {D. Lombardi and C. Pagliantini},
  title  = {{Conformal variational discretisation of infinite dimensional Hamiltonian systems with gradient flow dissipation}},
  note   = {arXiv:2412.06310},
  year   = {2024},
}

@article {AFW06,
    AUTHOR = {Arnold, D. N. and Falk, R. S. and Winther, R.},
     TITLE = {Finite element exterior calculus, homological techniques, and
              applications},
   JOURNAL = {Acta Numer.},
  FJOURNAL = {Acta Numerica},
    VOLUME = {15},
      YEAR = {2006},
     PAGES = {1--155},
      ISSN = {0962-4929,1474-0508},
      ISBN = {0-521-86815-7},
   MRCLASS = {58A15 (47N40 65N30 74G15 74S05)},
  MRNUMBER = {2269741},
MRREVIEWER = {Thomas\ Garrity},
}

@article {HughesFrancaMallet_86I,
    AUTHOR = {Hughes, T. J. R. and Franca, L. P. and Mallet, M.},
     TITLE = {A new finite element formulation for computational fluid
              dynamics. {I}. {S}ymmetric forms of the compressible {E}uler
              and {N}avier-{S}tokes equations and the second law of
              thermodynamics},
   JOURNAL = {Comput. Methods Appl. Mech. Engrg.},
  FJOURNAL = {Computer Methods in Applied Mechanics and Engineering},
    VOLUME = {54},
      YEAR = {1986},
    NUMBER = {2},
     PAGES = {223--234},
      ISSN = {0045-7825,1879-2138},
   MRCLASS = {76-08 (76N10)},
  MRNUMBER = {831553},
}

@article {HiltebrandMishra_spacetimeDG,
    AUTHOR = {Hiltebrand, A. and Mishra, S.},
     TITLE = {Entropy stable shock capturing space-time discontinuous
              {G}alerkin schemes for systems of conservation laws},
   JOURNAL = {Numer. Math.},
  FJOURNAL = {Numerische Mathematik},
    VOLUME = {126},
      YEAR = {2014},
    NUMBER = {1},
     PAGES = {103--151},
      ISSN = {0029-599X,0945-3245},
   MRCLASS = {65M60 (35L65 65M12)},
  MRNUMBER = {3149074},
MRREVIEWER = {Manfred\ Dobrowolski},
}

@article {ArnoldIP,
    AUTHOR = {Arnold, D. N.},
     TITLE = {An interior penalty finite element method with discontinuous
              elements},
   JOURNAL = {SIAM J. Numer. Anal.},
  FJOURNAL = {SIAM Journal on Numerical Analysis},
    VOLUME = {19},
      YEAR = {1982},
    NUMBER = {4},
     PAGES = {742--760},
      ISSN = {0036-1429},
   MRCLASS = {65N30},
  MRNUMBER = {664882},
}

@Article{ChenShuReview,
  author  = {T. Chen and C.-W. Shu},
  title   = {{Review of entropy stable discontinuous Galerkin methods for systems of conservation laws on unstructured simplex meshes}},
  journal = {CSIAM Trans. Appl. Math.},
  volume  = {1},
  pages   = {1--52},
  year    = {2020},
}

@article {OE_entropystable,
    AUTHOR = {Liu, Y. and Lu, J. and Shu, C.-W.},
     TITLE = {An entropy stable essentially oscillation-free discontinuous
              {G}alerkin method for hyperbolic conservation laws},
   JOURNAL = {SIAM J. Sci. Comput.},
  FJOURNAL = {SIAM Journal on Scientific Computing},
    VOLUME = {46},
      YEAR = {2024},
    NUMBER = {2},
     PAGES = {A1132--A1159},
      ISSN = {1064-8275,1095-7197},
   MRCLASS = {65M12 (65M60 76M20)},
  MRNUMBER = {4726005},
MRREVIEWER = {Xinghui\ Zhong},
}

@article {OE_systems,
    AUTHOR = {Liu, Y. and Lu, J. and Shu, C.-W.},
     TITLE = {An essentially oscillation-free discontinuous {G}alerkin
              method for hyperbolic systems},
   JOURNAL = {SIAM J. Sci. Comput.},
  FJOURNAL = {SIAM Journal on Scientific Computing},
    VOLUME = {44},
      YEAR = {2022},
    NUMBER = {1},
     PAGES = {A230--A259},
      ISSN = {1064-8275,1095-7197},
   MRCLASS = {65M60 (35L03 76Nxx)},
  MRNUMBER = {4372649},
MRREVIEWER = {Xiaobing\ Henry\ Feng},
}

@article {OE_conservationlaws,
    AUTHOR = {Lu, J. and Liu, Y. and Shu, C.-W.},
     TITLE = {An oscillation-free discontinuous {G}alerkin method for scalar
              hyperbolic conservation laws},
   JOURNAL = {SIAM J. Numer. Anal.},
  FJOURNAL = {SIAM Journal on Numerical Analysis},
    VOLUME = {59},
      YEAR = {2021},
    NUMBER = {3},
     PAGES = {1299--1324},
      ISSN = {0036-1429,1095-7170},
   MRCLASS = {65M60 (35L65 65M12)},
  MRNUMBER = {4257875},
}

@article {BeiraoLopezVacca_HamiltonianWave2d,
    AUTHOR = {Beir\~ao da Veiga, L. and Lopez, L. and Vacca, G.},
     TITLE = {Mimetic finite difference methods for {H}amiltonian wave
              equations in 2{D}},
   JOURNAL = {Comput. Math. Appl.},
  FJOURNAL = {Computers \& Mathematics with Applications. An International
              Journal},
    VOLUME = {74},
      YEAR = {2017},
    NUMBER = {5},
     PAGES = {1123--1141},
      ISSN = {0898-1221,1873-7668},
   MRCLASS = {65M06 (35L20 65M12)},
  MRNUMBER = {3689940},
MRREVIEWER = {Jose\ Alberto\ Cuminato},
}

@Article{Perrier_contraintDG,
  author  = {V. Perrier},
  title   = {{Development of discontinuous Galerkin methods for hyperbolic systems that preserve a curl or a divergence constraint: the case of linear systems}},
  journal = {J. Comput. Phys.},
  volume  = {544},
  pages   = {114445},
  year    = {2026},
}

@article {SanchezSymplecticMaxwell,
    AUTHOR = {S\'anchez, M. A. and Du, S. and Cockburn, B. and Nguyen, N.-C. and Peraire, J.},
     TITLE = {Symplectic {H}amiltonian finite element methods for
              electromagnetics},
   JOURNAL = {Comput. Methods Appl. Mech. Engrg.},
  FJOURNAL = {Computer Methods in Applied Mechanics and Engineering},
    VOLUME = {396},
      YEAR = {2022},
     PAGES = {114969},
      ISSN = {0045-7825,1879-2138},
   MRCLASS = {65M60 (74H15 74J05 74S05 78M10)},
   MRNUMBER = {4425194},
MRREVIEWER = {Agustin\ Mart\'in},
}

@article {AndrewsFarrell25,
    AUTHOR = {Andrews, B. D. and Farrell, P. E.},
     TITLE = {Enforcing conservation laws and dissipation inequalities
              numerically via auxiliary variables},
   JOURNAL = {SIAM J. Sci. Comput.},
  FJOURNAL = {SIAM Journal on Scientific Computing},
    VOLUME = {47},
      YEAR = {2025},
    NUMBER = {6},
     PAGES = {A3516--A3535},
      ISSN = {1064-8275,1095-7197},
   MRCLASS = {65M22 (35L65 65M60)},
  MRNUMBER = {4989363},
MRREVIEWER = {Andrea\ Villa},
}

@article {ArnoldBrezzi85,
    AUTHOR = {Arnold, D. N. and Brezzi, F.},
     TITLE = {Mixed and nonconforming finite element methods:
              implementation, postprocessing and error estimates},
   JOURNAL = {RAIRO Mod\'el. Math. Anal. Num\'er.},
  FJOURNAL = {RAIRO Mod\'elisation Math\'ematique et Analyse Num\'erique},
    VOLUME = {19},
      YEAR = {1985},
    NUMBER = {1},
     PAGES = {7--32},
      ISSN = {0764-583X,1290-3841},
   MRCLASS = {65N30},
  MRNUMBER = {813687},
}

@article {BKSternMaxwell,
    AUTHOR = {Berchenko-Kogan, Y. and Stern, A.},
     TITLE = {Constraint-preserving hybrid finite element methods for
              {M}axwell's equations},
   JOURNAL = {Found. Comput. Math.},
  FJOURNAL = {Foundations of Computational Mathematics. The Journal of the
              Society for the Foundations of Computational Mathematics},
    VOLUME = {21},
      YEAR = {2021},
    NUMBER = {4},
     PAGES = {1075--1098},
      ISSN = {1615-3375,1615-3383},
   MRCLASS = {78M10 (65M60 65N30)},
  MRNUMBER = {4298240},
MRREVIEWER = {Ziqiang\ Wang},
}

@article {BKSternYangMills,
    AUTHOR = {Berchenko-Kogan, Y. and Stern, A.},
     TITLE = {Charge-conserving hybrid methods for the {Y}ang-{M}ills
              equations},
   JOURNAL = {SMAI J. Comput. Math.},
  FJOURNAL = {SMAI Journal of Computational Mathematics},
    VOLUME = {7},
      YEAR = {2021},
     PAGES = {97--119},
      ISSN = {2426-8399},
   MRCLASS = {65M60 (65M55 78A25)},
  MRNUMBER = {4290513},
MRREVIEWER = {JiChun\ Li},
}

@article {McLachlanSternMS,
    AUTHOR = {McLachlan, R. I. and Stern, A.},
     TITLE = {Multisymplecticity of hybridizable discontinuous {G}alerkin
              methods},
   JOURNAL = {Found. Comput. Math.},
  FJOURNAL = {Foundations of Computational Mathematics. The Journal of the
              Society for the Foundations of Computational Mathematics},
    VOLUME = {20},
      YEAR = {2020},
    NUMBER = {1},
     PAGES = {35--69},
      ISSN = {1615-3375,1615-3383},
   MRCLASS = {65N30},
  MRNUMBER = {4056925},
}

@article {SternZampa26,
    AUTHOR = {Stern, A. and Zampa, E.},
     TITLE = {Multisymplecticity in {F}inite {E}lement {E}xterior
              {C}alculus},
   JOURNAL = {Found. Comput. Math.},
  FJOURNAL = {Foundations of Computational Mathematics. The Journal of the
              Society for the Foundations of Computational Mathematics},
    VOLUME = {26},
      YEAR = {2026},
    NUMBER = {3},
     PAGES = {1809--1858},
      ISSN = {1615-3375,1615-3383},
   MRCLASS = {70G45 (65N30)},
  MRNUMBER = {5010385},
}

@book {EGIII,
    AUTHOR = {Ern, A. and Guermond, J.-L.},
     TITLE = {Finite elements {III}---first-order and time-dependent {PDE}s},
    SERIES = {Texts Appl. Math.},
    VOLUME = {74},
 PUBLISHER = {Springer, Cham},
      YEAR = {2021},
     PAGES = {viii+417},
      ISBN = {978-3-030-57347-8; 978-3-030-57348-5},
   MRCLASS = {65-01},
  MRNUMBER = {4248810},
}

@article {SanchezSymplecticWave,
    AUTHOR = {S\'anchez, M. A. and Ciuca, C. and Nguyen, N. C. and Peraire,
              J. and Cockburn, B.},
     TITLE = {Symplectic {H}amiltonian {HDG} methods for wave propagation
              phenomena},
   JOURNAL = {J. Comput. Phys.},
  FJOURNAL = {Journal of Computational Physics},
    VOLUME = {350},
      YEAR = {2017},
     PAGES = {951--973},
      ISSN = {0021-9991,1090-2716},
   MRCLASS = {65M60 (35L20 65M12)},
  MRNUMBER = {3707190},
}

@article {SanchezSymplecticElastodynamics,
    AUTHOR = {S\'anchez, M. A. and Cockburn, B. and Nguyen, N.-C. and Peraire, J.},
     TITLE = {Symplectic {H}amiltonian finite element methods for linear
              elastodynamics},
   JOURNAL = {Comput. Methods Appl. Mech. Engrg.},
  FJOURNAL = {Computer Methods in Applied Mechanics and Engineering},
    VOLUME = {381},
      YEAR = {2021},
     PAGES = {113843},
      ISSN = {0045-7825,1879-2138},
   MRCLASS = {65M60 (74H15 74J05 74S05)},
  MRNUMBER = {4248352},
}

@article {SanchezSymplecticSemilinearwave,
    AUTHOR = {S\'anchez, M. A. and Valenzuela, J.},
     TITLE = {Symplectic {H}amiltonian finite element methods for semilinear
              wave propagation},
   JOURNAL = {J. Sci. Comput.},
  FJOURNAL = {Journal of Scientific Computing},
    VOLUME = {99},
      YEAR = {2024},
    NUMBER = {3},
     PAGES = {62},
      ISSN = {0885-7474,1573-7691},
   MRCLASS = {65M60 (65P10)},
  MRNUMBER = {4735361},
MRREVIEWER = {Hendrik\ Ranocha},
}

@article {EggerDissipative,
    AUTHOR = {Egger, H.},
     TITLE = {Structure preserving approximation of dissipative evolution
              problems},
   JOURNAL = {Numer. Math.},
  FJOURNAL = {Numerische Mathematik},
    VOLUME = {143},
      YEAR = {2019},
    NUMBER = {1},
     PAGES = {85--106},
      ISSN = {0029-599X,0945-3245},
   MRCLASS = {65M60 (37K05 37L65 47J35)},
  MRNUMBER = {3987168},
}

@Misc{AndrewsFarrell26,
  author = {B. D. Andrews and P. E. Farrell},
  title  = {{Conservative and dissipative discretisations of multi-conservative ODEs and GENERIC systems}},
  note   = {arXiv:2511.23266},
  year   = {2025},
}

@Article{ChanEntropyStableCorrectionDG,
  author  = {J. Chan},
  title   = {{An artificial viscosity approach to high order entropy stable discontinuous Galerkin methods}},
  journal = {J. Comput. Phys.},
  volume  = {543},
  pages   = {114380},
  year    = {2025},
}

@Misc{SternZampa27,
  author = {A. Stern and E. Zampa},
  title  = {{Finite element exterior calculus for time-dependent Hamiltonian partial differential equations}},
  note   = {arXiv:2601.00103},
  year   = {2026},
}

@article {AbgrallCorrection,
    AUTHOR = {Abgrall, R.},
     TITLE = {A general framework to construct schemes satisfying additional
              conservation relations. {A}pplication to entropy conservative
              and entropy dissipative schemes},
   JOURNAL = {J. Comput. Phys.},
  FJOURNAL = {Journal of Computational Physics},
    VOLUME = {372},
      YEAR = {2018},
     PAGES = {640--666},
      ISSN = {0021-9991,1090-2716},
   MRCLASS = {65N30},
  MRNUMBER = {3847450},
}

@article {BustoDumbserNewClass,
    AUTHOR = {Busto, S. and Dumbser, M.},
     TITLE = {A new class of simple, general and efficient finite volume
              schemes for overdetermined thermodynamically compatible
              hyperbolic systems},
   JOURNAL = {Commun. Appl. Math. Comput.},
  FJOURNAL = {Communications on Applied Mathematics and Computation},
    VOLUME = {6},
      YEAR = {2024},
    NUMBER = {3},
     PAGES = {1742--1778},
      ISSN = {2096-6385,2661-8893},
   MRCLASS = {35L40 (65M08)},
  MRNUMBER = {4772608},
MRREVIEWER = {Dan\ Ling},
}

@Article{HTCGPR,
  author  = {S. Busto and M. Dumbser and I. Peshkov and E. Romenski},
  title   = {{On thermodynamically compatible finite volume schemes for continuum mechanics}},
  journal = {SIAM J. Sci. Comput.},
  volume  = {44},
  pages   = {A1723--A1751},
  year    = {2022},
}

@article {BustoDumbserMHD,
    AUTHOR = {Busto, S. and Dumbser, M.},
     TITLE = {A new thermodynamically compatible finite volume scheme for
              magnetohydrodynamics},
   JOURNAL = {SIAM J. Numer. Anal.},
  FJOURNAL = {SIAM Journal on Numerical Analysis},
    VOLUME = {61},
      YEAR = {2023},
    NUMBER = {1},
     PAGES = {343--364},
      ISSN = {0036-1429,1095-7170},
   MRCLASS = {65M08 (35L40 76M12 76W05)},
  MRNUMBER = {4550689},
}

@article {AbgrallBustoDumbserGeneralFramework,
    AUTHOR = {Abgrall, R. and Busto, S. and Dumbser, M.},
     TITLE = {A simple and general framework for the construction of
              thermodynamically compatible schemes for computational fluid
              and solid mechanics},
   JOURNAL = {Appl. Math. Comput.},
  FJOURNAL = {Applied Mathematics and Computation},
    VOLUME = {440},
      YEAR = {2023},
     PAGES = {127629},
      ISSN = {0096-3003,1873-5649},
   MRCLASS = {65M08 (65M12 65M60 76B15)},
  MRNUMBER = {4514646},
}

@article {RebholzHelicity,
    AUTHOR = {Rebholz, L. G.},
     TITLE = {An energy- and helicity-conserving finite element scheme for
              the {N}avier-{S}tokes equations},
   JOURNAL = {SIAM J. Numer. Anal.},
  FJOURNAL = {SIAM Journal on Numerical Analysis},
    VOLUME = {45},
      YEAR = {2007},
    NUMBER = {4},
     PAGES = {1622--1638},
      ISSN = {0036-1429,1095-7170},
   MRCLASS = {65M60 (76D05 76M10)},
  MRNUMBER = {2338402},
MRREVIEWER = {Srinivasan\ Kesavan},
}

@incollection {Barth99,
    AUTHOR = {Barth, T. J.},
     TITLE = {Numerical methods for gasdynamic systems on unstructured
              meshes},
 BOOKTITLE = {An introduction to recent developments in theory and numerics
              for conservation laws ({F}reiburg/{L}ittenweiler, 1997)},
    SERIES = {Lect. Notes Comput. Sci. Eng.},
    VOLUME = {5},
     PAGES = {195--285},
 PUBLISHER = {Springer, Berlin},
      YEAR = {1999},
      ISBN = {3-540-65081-4},
   MRCLASS = {76M10 (35L65 65M60 76M12 76N15)},
  MRNUMBER = {1731618},
MRREVIEWER = {Jean-Luc\ Guermond},
}

@InProceedings{ZakerzadehMay,
  author    = {M. Zakerzadeh and G. May},
  title     = {{Entropy stable discontinuous Galerkin scheme for the compressible Navier--Stokes equations}},
  booktitle = {55th AIAA Aerospace Sciences Meeting},
  note      = {AIAA Paper 2017-0084},
  year      = {2017},
}

@article {Chan18,
    AUTHOR = {Chan, J.},
     TITLE = {On discretely entropy conservative and entropy stable
              discontinuous {G}alerkin methods},
   JOURNAL = {J. Comput. Phys.},
  FJOURNAL = {Journal of Computational Physics},
    VOLUME = {362},
      YEAR = {2018},
     PAGES = {346--374},
      ISSN = {0021-9991,1090-2716},
   MRCLASS = {65M60 (65M12 76M10 76Nxx)},
  MRNUMBER = {3774937},
MRREVIEWER = {Mir\ Sajjad\ Hashemi},
}

@article {Ortleb2017,
    AUTHOR = {Ortleb, S.},
     TITLE = {A kinetic energy preserving {DG} scheme based on
              {G}auss-{L}egendre points},
   JOURNAL = {J. Sci. Comput.},
  FJOURNAL = {Journal of Scientific Computing},
    VOLUME = {71},
      YEAR = {2017},
    NUMBER = {3},
     PAGES = {1135--1168},
      ISSN = {0885-7474,1573-7691},
   MRCLASS = {65M60 (35L65 65M08 76P05 82C40)},
  MRNUMBER = {3640675},
}

@article {RanochaOffnerSonar2017,
    AUTHOR = {Ranocha, H. and \"Offner, P. and Sonar, T.},
     TITLE = {Extended skew-symmetric form for summation-by-parts operators
              and varying {J}acobians},
   JOURNAL = {J. Comput. Phys.},
  FJOURNAL = {Journal of Computational Physics},
    VOLUME = {342},
      YEAR = {2017},
     PAGES = {13--28},
      ISSN = {0021-9991,1090-2716},
   MRCLASS = {65M06 (65M12)},
  MRNUMBER = {3649262},
}

@article{Ranocha17,
  title = {Shallow water equations: Split-form, entropy stable, well-balanced, and positivity preserving numerical methods},
  author = {H. Ranocha},
  journal = {GEM Int. J. Geomath.},
  year = {2017},
  volume = {8},
  number = {1},
  pages = {85--133},
}

@article {ChenShu17,
    AUTHOR = {Chen, T. and Shu, C.-W.},
     TITLE = {Entropy stable high order discontinuous {G}alerkin methods
              with suitable quadrature rules for hyperbolic conservation
              laws},
   JOURNAL = {J. Comput. Phys.},
  FJOURNAL = {Journal of Computational Physics},
    VOLUME = {345},
      YEAR = {2017},
     PAGES = {427--461},
      ISSN = {0021-9991,1090-2716},
   MRCLASS = {65M60 (65M12)},
  MRNUMBER = {3667622},
MRREVIEWER = {Gerhard\ Starke},
}

@article {DGMaxwell_PipernoEtAl,
    AUTHOR = {Fezoui, L. and Lanteri, S. and Lohrengel, S. and Piperno, S.},
     TITLE = {Convergence and stability of a discontinuous {G}alerkin
              time-domain method for the 3{D} heterogeneous {M}axwell
              equations on unstructured meshes},
   JOURNAL = {M2AN Math. Model. Numer. Anal.},
  FJOURNAL = {M2AN. Mathematical Modelling and Numerical Analysis},
    VOLUME = {39},
      YEAR = {2005},
    NUMBER = {6},
     PAGES = {1149--1176},
      ISSN = {0764-583X,1290-3841},
   MRCLASS = {65M60 (78M10)},
  MRNUMBER = {2195908},
MRREVIEWER = {Lech\ Slawik},
}

@TechReport{ngSolve,
  author      = {J. Sch{\"o}berl},
  title       = {{C++11 implementation of finite elements in NGSolve}},
  institution = {Institute for Analysis and Scientific Computing, TU Wien},
  number      = {ASC-2014-30},
  type        = {Technical Report},
  year        = {2014},
}

@article {DGlinearPH,
    AUTHOR = {Cheng, X. and van der Vegt, J. J. W. and Xu, Y. and Zwart, H. J.},
     TITLE = {Discontinuous {G}alerkin finite element methods for linear
              port-{H}amiltonian dynamical systems},
   JOURNAL = {J. Sci. Comput.},
  FJOURNAL = {Journal of Scientific Computing},
    VOLUME = {104},
      YEAR = {2025},
    NUMBER = {1},
     PAGES = {8},
      ISSN = {0885-7474,1573-7691},
   MRCLASS = {65M60 (65M12)},
  MRNUMBER = {4907637},
}

@article {DGPH,
    AUTHOR = {Kumar, N. and van der Vegt, J. J. W. and Zwart, H. J.},
     TITLE = {Port-{H}amiltonian discontinuous {G}alerkin finite element
              methods},
   JOURNAL = {IMA J. Numer. Anal.},
  FJOURNAL = {IMA Journal of Numerical Analysis},
    VOLUME = {45},
      YEAR = {2025},
    NUMBER = {1},
     PAGES = {354--403},
      ISSN = {0272-4979,1464-3642},
   MRCLASS = {65N30 (58A10 65N12)},
  MRNUMBER = {4858500},
MRREVIEWER = {M\'onika\ Polner},
}

@Article{Tadmor1,
  author  = {E. Tadmor},
  journal = {Math. Comp.},
  title   = {{The numerical viscosity of entropy stable schemes for systems of conservation laws I}},
  year    = {1987},
  pages   = {91--103},
  volume  = {49},
}

@Article{Tadmor2003,
  author = 	 {E. Tadmor},
  title = 	 {Entropy stability theory for difference approximations of nonlinear conservation laws and related time dependent problems},
  journal = 	 {Acta Numer.},
  year = 	 2003,
  volume = 	 12,
  pages = 	 {451--512}}

@Article{HTCDG,
  author   = {S. Busto and M. Dumbser},
  journal  = {J. Sci. Comput.},
  title    = {{A new family of thermodynamically compatible discontinuous Galerkin methods for continuum mechanics and turbulent shallow water flows}},
  year     = {2022},
  number   = {2},
  pages    = {56},
  volume   = {93},
  fjournal = {Journal of Scientific Computing},
}

@Article{FisherCarpenter2013,
  author  = {T.C. Fisher and M.H. Carpenter},
  title   = {{High-order entropy stable finite difference schemes for nonlinear conservation laws: Finite domains}},
  journal = {J. Comput. Phys.},
  volume  = {252},
  pages   = {518--557},
  year    = {2013},
}

@Article{CarpenterFisherNielsenFrankel2014,
  author  = {M. Carpenter and T. Fisher and E. Nielsen and S. Frankel},
  title   = {{Entropy stable spectral collocation schemes for the Navier--Stokes equations: Discontinuous interfaces}},
  journal = {SIAM J. Sci. Comput.},
  volume  = {36},
  pages   = {B835--B867},
  year    = {2014},
}

@Article{Gassner2013,
  author  = {G.J. Gassner},
  title   = {{A skew-symmetric discontinuous Galerkin spectral element discretization and its relation to SBP-SAT finite difference methods}},
  journal = {SIAM J. Sci. Comput.},
  volume  = {35},
  pages   = {A1233--A1253},
  year    = {2013},
}

@Article{GassnerWintersKopriva2016,
  author  = {G.J. Gassner and A.R. Winters and D.A. Kopriva},
  title   = {{Split form nodal discontinuous Galerkin schemes with summation-by-parts property for the compressible Euler equations}},
  journal = {J. Comput. Phys.},
  volume  = {327},
  pages   = {39--66},
  year    = {2016},
}

@Article{LuccaDumbser2026,
  author  = {A. Lucca and M. Dumbser},
  title   = {{Structure-preserving schemes for nonlinear symmetric hyperbolic and thermodynamically compatible systems of partial differential equations}},
  journal = {J. Sci. Comput.},
  volume  = {107},
  pages   = {16},
  year    = {2026},
}

@Article{DumbserLuccaPeshkovZanotti2025,
  author  = {M. Dumbser and A. Lucca and I. Peshkov and O. Zanotti},
  title   = {{Variational derivation and compatible discretizations of the Maxwell-GLM system}},
  journal = {Proc. R. Soc. A},
  volume  = {481},
  pages   = {20240864},
  year    = {2025},
}

@Article{Ruggeri81,
  author  = {T. Ruggeri and A. Strumia},
  title   = {{Main field and convex covariant density for quasilinear hyperbolic systems. Relativistic fluid dynamics}},
  journal = {Ann. Inst. H. Poincar\'e Sect. A (N.S.)},
  year    = {1981},
  volume  = {34},
  pages   = {65--84},
}

@Article{Ketcheson2019,
  author  = {D. I. Ketcheson},
  title   = {{Relaxation Runge--Kutta methods: conservation and stability for inner-product norms}},
  journal = {SIAM J. Numer. Anal.},
  volume  = {57},
  pages   = {2850--2870},
  year    = {2019},
}

@Article{RanochaEtAl2020,
  author  = {H. Ranocha and M. Sayyari and L. Dalcin and M. Parsani and D. I. Ketcheson},
  title   = {{Relaxation Runge--Kutta methods: fully discrete explicit entropy-stable schemes for the compressible Euler and Navier--Stokes equations}},
  journal = {SIAM J. Sci. Comput.},
  volume  = {42},
  pages   = {A612--A638},
  year    = {2020},
}

@Article{RanochaDalcinParsani2020,
  author  = {H. Ranocha and L. Dalcin and M. Parsani},
  title   = {{Fully discrete explicit locally entropy-stable schemes for the compressible Euler and Navier--Stokes equations}},
  journal = {Comput. Math. Appl.},
  volume  = {80},
  pages   = {1343--1359},
  year    = {2020},
}

@Article{Perrier2,
  author  = {V. Perrier},
  title   = {{Discrete de-Rham complex involving a discontinuous finite element space for velocities: the case of periodic straight triangular and Cartesian meshes}},
  journal = {Ann. Henri Lebesgue},
  volume  = {8},
  pages   = {417--452},
  year    = {2025},
}

@Article{Perrier3,
  author  = {J. Jung and V. Perrier},
  title   = {{A curl preserving finite volume scheme by space velocity enrichment. Application to the low Mach number accuracy problem}},
  journal = {J. Comput. Phys.},
  volume  = {515},
  pages   = {113252},
  year    = {2024},
}

@Article{HTCLagrangeGPR,
  author  = {W. Boscheri and M. Dumbser and R. Loub{\`e}re and P.-H. Maire},
  title   = {{A structure-preserving and thermodynamically compatible cell-centered Lagrangian finite volume scheme for continuum mechanics}},
  journal = {SIAM J. Numer. Anal.},
  volume  = {64},
  pages   = {685--707},
  year    = {2026},
}

@Article{ErnGuermond2023,
  author  = {A. Ern and J.-L. Guermond},
  title   = {{The discontinuous {G}alerkin approximation of the grad-div and curl-curl operators in first-order form is involution-preserving and spectrally correct}},
  journal = {SIAM J. Numer. Anal.},
  volume  = {61},
  pages   = {2940--2966},
  year    = {2023},
}

@Article{CompatibleDG1,
  author  = {R. Abgrall and M. Dumbser and P.-H. Maire},
  title   = {{A simple and general framework for the construction of exactly div-curl-grad compatible discontinuous Galerkin finite element schemes on unstructured simplex meshes}},
  journal = {J. Comput. Phys.},
  volume  = {541},
  pages   = {114340},
  year    = {2025},
}

@Article{CockburnLiShu2004,
  author  = {B. Cockburn and F. Li and C.-W. Shu},
  title   = {{Locally divergence-free discontinuous Galerkin methods for the Maxwell equations}},
  journal = {J. Comput. Phys.},
  volume  = {194},
  pages   = {588--610},
  year    = {2004},
}

@Misc{FourSplitUS,
  author = {M. Dumbser and W. Boscheri and M. Tavelli and A. Thomann},
  title  = {{A structure-preserving staggered semi-implicit four-split finite volume scheme for continuum mechanics on unstructured meshes}},
  note   = {arXiv:2609.25935},
  year   = {2026},
}

@Article{JiangShu1994,
  author  = {G.-S. Jiang and C.-W. Shu},
  title   = {{On a cell entropy inequality for discontinuous Galerkin methods}},
  journal = {Math. Comp.},
  volume  = {62},
  pages   = {531--538},
  year    = {1994},
}

@Article{FriedrichEtAl2019,
  author  = {L. Friedrich and G. Schn{\"u}cke and A. R. Winters and D. C. {Del Rey Fern{\'a}ndez} and G. J. Gassner and M. H. Carpenter},
  title   = {{Entropy stable space--time discontinuous Galerkin schemes with summation-by-parts property for hyperbolic conservation laws}},
  journal = {J. Sci. Comput.},
  volume  = {80},
  pages   = {175--222},
  year    = {2019},
}

@Article{GaburroEtAl2023,
  author  = {E. Gaburro and P. {\"O}ffner and M. Ricchiuto and D. Torlo},
  title   = {{High order entropy preserving ADER-DG schemes}},
  journal = {Appl. Math. Comput.},
  volume  = {440},
  pages   = {127644},
  year    = {2023},
}

@Article{BetschSteinmann2000,
  author  = {P. Betsch and P. Steinmann},
  title   = {{Inherently energy conserving time finite elements for classical mechanics}},
  journal = {J. Comput. Phys.},
  volume  = {160},
  pages   = {88--116},
  year    = {2000},
}

@Article{QuispelMcLaren2008,
  author  = {G. R. W. Quispel and D. I. McLaren},
  title   = {{A new class of energy-preserving numerical integration methods}},
  journal = {J. Phys. A: Math. Theor.},
  volume  = {41},
  pages   = {045206},
  year    = {2008},
}

@Article{Hairer2010,
  author  = {E. Hairer},
  title   = {{Energy-preserving variant of collocation methods}},
  journal = {JNAIAM J. Numer. Anal. Ind. Appl. Math.},
  volume  = {5},
  pages   = {73--84},
  year    = {2010},
}

@Book{BrugnanoIavernaro2016,
  author    = {L. Brugnano and F. Iavernaro},
  title     = {{Line Integral Methods for Conservative Problems}},
  publisher = {Chapman \& Hall/CRC},
  address   = {Boca Raton},
  year      = {2016},
}

@Article{BrugnanoIavernaro2018,
  author  = {L. Brugnano and F. Iavernaro},
  title   = {{Line integral solution of differential problems}},
  journal = {Axioms},
  volume  = {7},
  pages   = {36},
  year    = {2018},
}

@BOOK{toro-book, 
  AUTHOR =       {E. F. Toro}, 
  TITLE =        {{Riemann} Solvers and Numerical Methods for Fluid
                  Dynamics}, 
  PUBLISHER =    {Springer}, 
  YEAR =         {2009}, 
}

@article{GJP25, 
        author={S. G{\'o}mez and A. J{\"u}ngel and I. Perugia},
        title={{Structure-preserving Local Discontinuous Galerkin method for nonlinear cross-diffusion systems}},
        journal={IMA J. Numer. Anal.},
        pages={draf140},
        year={2026},
        month = {04}
}
\end{document}